\documentclass[]{interact}

\usepackage{epstopdf}

\usepackage[caption=false]{subfig}

\usepackage[numbers,sort&compress]{natbib}

\usepackage{algorithm}
\usepackage{algpseudocode}
\usepackage{hyperref}

\bibpunct[, ]{[}{]}{,}{n}{,}{,}
\makeatletter
\def\NAT@def@citea{\def\@citea{\NAT@separator}}
\makeatother

\theoremstyle{plain}
\newtheorem{theorem}{Theorem}[section]
\newtheorem{lemma}[theorem]{Lemma}
\newtheorem{Remark}[theorem]{Remark}
\newtheorem{corollary}[theorem]{Corollary}
\newtheorem{proposition}[theorem]{Proposition}
\newtheorem{Problem}[theorem]{Problem}
\theoremstyle{definition}

\theoremstyle{remark}
\newtheorem{remark}{Remark}

\newcommand{\R}{\mathbb{R}}

\newcommand{\calL}{\mathcal{L}}
\newcommand{\ds}{\displaystyle}
\newcommand{\Id}{{\rm Id}}

\newcommand{\xx}{{\bf x}}
\newcommand{\bb}{{\bf b}}
\newcommand{\rand}{{\rm rand}}

\newcommand{\Div}{{\rm Div}}

\newcommand{\calE}{\mathcal{E}}

\begin{document}


\title{Recovering the initial condition of parabolic equations from lateral Cauchy data via the quasi-reversibility method}

\author{
\name{Qitong Li\textsuperscript{a} and Loc Hoang Nguyen\textsuperscript{a}\thanks{CONTACT Loc Hoang Nguyen. Email: loc.nguyen@uncc.edu}}
\affil{\textsuperscript{a}Department of Mathematics and Statistics, University of North Carolina
Charlotte, Charlotte, NC, 28223, USA, qli13@uncc.edu, loc.nguyen@uncc.edu.}
}

\maketitle

\begin{abstract}
We consider the problem of computing the initial condition for a general parabolic equation from the Cauchy lateral data.
	The stability of this problem is well-known to be logarithmic.
	In this paper, we introduce an approximate model, as a coupled linear system of elliptic partial differential equations. Solution to this model is the vector of Fourier coefficients of the solutions to the parabolic equation above. 
	This approximate model is solved by the quasi-reversibility method. 
	We will prove the convergence for the quasi-reversibility method as the measurement noise tends to $0$.
	The convergent rate is Lipschitz.
	We present the implementation of our algorithm in details and verify our method by showing some numerical examples.
\end{abstract}


\section{Introduction}
\label{sec 1}

Let $d \geq 2$ be the spatial dimension and $\Omega$  be a open and bounded domain in $\R^d$. 
Assume that $\partial \Omega$ is smooth.
Let 
	\begin{equation}
		A = (a_{ij})_{i,j = 1}^d \in C^2(\R^d, \R^{d \times d}) \cap L^{\infty}(\R^d, \R^{d \times d})
	\label{A}
	\end{equation}
	 satisfy the following conditions
\begin{enumerate}
\item $A$ is symmetric; i.e, $A^T(\xx) = A(\xx)$  for all $\xx \in \R^d;$
\item $A$ is uniformly elliptic; i.e., there exists a positive number $\mu$ such that 
\begin{equation}
	  A(\xx) \xi \cdot \xi \geq \mu |\xi|^2 \quad \mbox{for all } \xx, \xi = (\xi_1, \dots, \xi_d) \in \R^d.
	\label{coercive}
\end{equation}
\end{enumerate}
Let $\bb = (b_1, b_2, \dots, b_d)\in C^1(\R^d, \R^d) \cap L^{\infty}(\R^d, \R^d)$ and $c \in C^1(\R^d, \R)\cap L^{\infty}(\R^d, \R)$.
Define 
\begin{equation}
	\calL v = \Div (A \nabla v) +  \bb(\xx) \cdot \nabla v(\xx) + c(\xx) v
	\label{calL}
\end{equation}
for all functions $v \in C^2(\R^d)$.
Consider the initial value problem
\begin{equation}
	\left\{
		\begin{array}{rcll}
			u_t(\xx, t) &=&\calL u(\xx, t) & \xx \in \R^d, t > 0\\
			u(\xx, 0) &=& f(\xx) &\xx \in \R^d
		\end{array}
	\right.
	\label{heat eqn}
\end{equation}
where $f \in L^2(\R^d)$ represents an initial source with support compactly contained in $\Omega$.  
We refer the reader to the books \cite{Evans:PDEs2010, Ladyzhenskaya:sv1985}.
The main aim of this paper is to solve the following problem.
\begin{Problem}
	Let $T > 0$. Given the Cauchy boundary data
	\begin{equation}
		F(\xx, t) = u(\xx, t) \quad \mbox{and } \quad G(\xx, t) = A \nabla u(\xx, t) \cdot \nu
		\label{data}
	\end{equation}
	for $ \xx \in \partial \Omega, t \in [0, T],$
	determine the function $f(\xx),$ $\xx \in \Omega$.
	Here $\nu$ is the outward normal to $\partial \Omega.$
	\label{ISP}
\end{Problem}

Problem \ref{ISP} is the problem of recovering the initial condition of  the parabolic equation from the lateral Cauchy data.
This problem has many real-world applications ; for e.g.,	 
determine the spatially distributed temperature inside a solid from the boundary measurement of the heat and heat flux in the time domain \cite{Klibanov:ip2006}; 
identify the pollution on the surface of the rivers or lakes \cite{BadiaDuong:jiip2002};
 effectively monitor the heat conductive processes in steel industries, glass and polymer forming and nuclear power station \cite{LiYamamotoZou:cpaa2009}.
Due to its realistic applications, this problem has been studied  intensively.
The uniqueness of Problem \ref{ISP} is well-known, see \cite{Lavrentiev:AMS1986}. 
Also, it can be reduced from the logarithmic stability results in \cite{Klibanov:ip2006, LiYamamotoZou:cpaa2009}. 
The natural approach to solve this problem is the optimal control method; that means, minimize a mismatch functional. 
The proof of the convergence of the optimal control method to the true solution to these inverse problems is challenging and is omitted.
One of our contributions to the field is the convergence of the quasi-reversibility method, which our method is relied on, as the measurement noise tends to $0$.

Related to the inverse problem in the current paper, the  problem of recovering the initial conditions for hyperbolic equation is very interesting since it arises in many real-world applications. 
For instance the problems thermo and photo acoustic tomography play the key roles in bio-medical imaging. 
We refer the reader to some important works in this field \cite{LiuUhlmann:ip2015, KatsnelsonNguyen:aml2018, HaltmeierNguyen:SIAMJIS2017}. 
Applying the Fourier transform, one can reduce the problem of reconstructing the initial conditions for hyperbolic equations to some inverse source problems for the Helmholtz equation, see \cite{NguyenLiKlibanov:IPI2019, WangGuoZhangLiu:ip2017, WangGuoLiLiu:ip2017, LiLiuSun:IPI2018, ZhangGuoLiu:ip2018} for some recent results.

In this paper, we employ the technique developed by our own research group. 
The main point of this technique is to derive an approximate model for the Fourier coefficients of the solution to the governing partial differential equation. 
This technique was first introduced in \cite{Klibanov:jiip2017}.
This approximate model is a system of elliptic equations. 
It, together with Cauchy boundary data, is solved by the quasi-reversibility method. 
This approach was used to solve an inverse source problem for Helmholtz equation \cite{NguyenLiKlibanov:IPI2019} and to inverse the Radon transform with incomplete data \cite{KlibanovNguyen:ip2019}.
Especially, Klibanov, Li and Zhang \cite{KlibanovLiZhang:ip2019} used  the convexification method, a stronger version of this technique, to compute numerical solutions to the nonlinear problem of electrical impedance tomography with restricted Dirichlet-to-Neumann map data.
It is remarkable mentioning that the numerical solutions in \cite{KlibanovLiZhang:ip2019} due to the convexification method are impressive.

As mentioned in the previous paragraph, we employ the quasi-reversibility method to solve an approximate model for Fourier coefficients of the solution to \eqref{heat eqn}.
This method was first introduced by Latt\'es and Lions \cite{LattesLions:e1969}. 
It is used to computed numerical solutions to ill-posed problems for partial differential equations. 
Due to its strength, since then, the quasi-reversibility method attracts the great attention of the scientific community
see e.g., \cite{Becacheelal:AIMS2015, Bourgeois:ip2006,
BourgeoisDarde:ip2010, BourgeoisPonomarevDarde:ipi2019, ClasonKlibanov:sjsc2007, Dadre:ipi2016, KaltenbacherRundell:ipi2019,
KlibanovSantosa:SIAMJAM1991, Klibanov:jiipp2013, LocNguyen:ip2019}.
We refer the reader to \cite{Klibanov:anm2015} for a survey
on this method. The solution of the
approximate model in the previous paragraph due to the quasi-reversibility method is called
{\it regularized solution} in the theory of ill-posed problems \cite{Tihkonov:kapg1995}. 
A question arises immediately about the convergence of the quasi-reversibility method:
whether or not the regularized solutions obtained by the quasi-reversibility method converges to the true solution of our system of partial differential equations as the noise tends to $0$.
The affirmative answer to this question is obtained using a general Carleman estimate.
Moreover, we employ a Carleman estimate to prove that the convergence rate is Lipschitz. 
It is important mentioning that in the celebrate paper \cite{BukhgeimKlibanov:smd1981}, Bukhgeim and Klibanov discovered the use of Carleman estimate in studying inverse problems for all three main types of partial differential equations.

The paper is organized as follows. 
In Section \ref{sec ideas}, we describe our approach and propose an algorithm to solve Problem \ref{ISP}.
In Section \ref{sec Carleman}, we employ prove a Carleman estimate.
Then, in Section \ref{sec convergence}, we study the convergence of the quasi-reversibility method as the noise tends to $0$.
Finally, in Section \ref{sec num}, we present all details about the numerical implementation and then show some numerical results from highly noisy simulated data.

\section{The algorithm to solve Problem \ref{ISP}} \label{sec ideas}

We will employ the following basis to introduce an approximation model.

\subsection{An orthonormal basis of $L^2(0, T)$  and the truncated Fourier series} \label{sec basis}

For each $n > 1$, define a complete sequence $\{\phi_n\}_{n = 1}^{\infty}$ in $L^2(0, T)$ with 
\begin{equation}
	\phi_n(t) = (t - t_0)^{n - 1} \exp(t - t_0)
	\label{phi}
\end{equation} 
	 where $t_0 = T/2.$	
Using the Gram-Schmidt orthonormalization for the sequence $\{\phi_n\}_{n = 1}^{\infty}$, we can construct an orthonormal basis of $L^2(0, T),$ named as $\{\Psi_n\}_{n = 1}^{\infty}$. For each $n,$ the function $\Psi_n(t)$ takes the form 
\begin{equation}
	\Psi_n(t) = P_{n - 1}(t - t_0) \exp(t - t_0)
	\label{basis}
\end{equation}
where $P_{n - 1}$ is a polynomial of the $(n - 1)^{\rm th}$ order.
For each $\xx \in \Omega,$ we consider $u(\xx, \cdot)$ as a function with respect to $t$. The Fourier series of this function is 
\begin{equation}
	u(\xx, t) = \sum_{n = 1}^\infty u_n(\xx) \Psi_n(t) 
	\label{Fourier series}
\end{equation}
where
\begin{equation}
	u_n({\bf x}) = \int_{0}^{T} u({\bf x}, t) \Psi_n(t)dt, \quad n = 1, 2, \dots.
	\label{2.2}
\end{equation}
Fix a positive integer $N$. 
We truncate the Fourier series in \eqref{Fourier series}.
 The function $u(\xx, t)$ is approximated by
\begin{equation}
	u({\bf x}, t) = \sum_{n = 1}^N u_n({\bf x}) \Psi_n(t)
	\quad \xx \in \Omega, t \in [0, T].
	\label{2.1111}
\end{equation}
In this context, the partial derivative with respect to $t$ of $u({\bf x}, t)$ is approximated by
\begin{equation}
	 u_t({\bf x}, t) = \sum_{n = 1}^N u_n({\bf x}) \Psi'_n(t) 
	\label{2.1}
\end{equation}
for all ${\bf x} \in \Omega$ and $t \in (0, T).$

To reconstruct the wave field $u({\bf x}, t)$, we compute the Fourier coefficients $u_n({\bf x})$, $1 \leq n \leq N$. 
It is obvious that \eqref{2.1111} and \eqref{2.1} play crucial roles in this step. 
We; therefore, require that the function $\Psi'_n$ cannot be identically $0$. 
The usual ``sin and cosine" basis of the Fourier transform does not meet this requirement while it is not hard to verify from \eqref{basis} that the basis $\{\Psi_n\}_{n = 1}^{\infty}$ does.
The basis $\{\Psi_n\}_{n = 1}^{\infty}$ was first introduced in \cite{Klibanov:jiip2017}. 
Then, this basis was successfully used to solve several important inverse problems, including the inverse source problem for Helmholtz equations \cite{NguyenLiKlibanov:IPI2019}, inverse X-ray tomographic problem in incomplete data \cite{KlibanovNguyen:ip2019} and the nonlinear inverse problem of electrical impedance tomography with restricted Dirichlet to Neumann map data, see \cite{KlibanovLiZhang:ip2019}.


\subsection{An approximate model}

We introduce in this subsection a coupled system of elliptic partial differential equations without the presence of the unknown function $f(\xx)$.
Plugging \eqref{2.1111} and \eqref{2.1} into \eqref{heat eqn},
we have
\begin{equation}
	\sum_{n = 1}^N u_n({\bf x}) \Psi'_n(t) = \sum_{n = 1}^N \calL u_n(\xx) \Psi_n(t).
	\label{2.3}
\end{equation}
for all $\xx \in \Omega$ and $t \in [0, T].$
For each $m \in \{1, \dots, N\}$, multiply $\Psi_m(t)$ to both sides of \eqref{2.3} and then integrating the obtained equation with respect to $t$, we obtain
\begin{equation}
	\sum_{n = 1}^N  \left(\int_0^T\Psi_m(t)\Psi_n'(t)dt\right) u_n(\xx)
	= \sum_{n = 1}^N\left(\int_0^T\Psi_m(t)\Psi_n(t)dt\right) \calL u_n(\xx)	\label{2.4}
\end{equation}
for all $\xx$ in $\Omega.$
Denote by
\begin{equation}
	s_{mn} = \int_0^T \Psi_m(t)\Psi_n'(t) dt
	\label{matrix S}
\end{equation}
and note that
\begin{equation}
	\int_0^T \Psi_m(t)\Psi_n(t) dt 
	= 
	\left\{
		\begin{array}{ll}
			1 &\mbox{if } m = n,\\
			0 &\mbox{if } m \not = n.
		\end{array}	
	\right.
\end{equation}
We rewrite \eqref{2.4} as
\begin{equation}
	\sum_{n = 1}^N s_{mn} u_n(\xx)  = \calL u_m(\xx) \quad \xx \in \Omega, m = 1, 2, \dots, N.
	\label{2.9}
\end{equation}
Denote 
\begin{equation}
	U(\xx) = (u_1(\xx), u_2(\xx), \dots, u_N(\xx))^T.
	\label{capU}
\end{equation} 
It follows from \eqref{2.9} that
\begin{equation}
	 \calL U(\xx) = S U(\xx) \quad \mbox{for } \xx \in \Omega
	 \label{2.10}
\end{equation}
where $S$ is the $N \times N$ matrix whose $mn^{\rm th}$ entry is given in \eqref{matrix S}, $1 \leq m, n \leq N$.
Here, the operator $\calL$ acting on the vector $U(\xx)$ is understood in the same manner as it acts on scalar valued function, see \eqref{calL}.

On the other hand, due to \eqref{2.2} and \eqref{data},
the vector $U(\xx)$ satisfies the boundary conditions
\begin{align}
	U(\xx) &= \widetilde F(\xx) = \left(\int_{0}^T F(\xx, t) \Psi_1(t) dt,  \dots, \int_{0}^T F(\xx, t) \Psi_N(t) dt \right)^T
	\label{2.11}\\
	A \nabla U(\xx) \cdot \nu &= \widetilde G(\xx) = \left(\int_{0}^T G(\xx, t) \Psi_1(t) dt,  \dots, \int_{0}^T G(\xx, t) \Psi_N(t) dt \right)^T
	\label{2.12}
\end{align}
for all $\xx \in \partial \Omega.$

\begin{Remark}
	From now on, we consider $\widetilde F$ and $\widetilde G$ as our ``indirect" boundary data. This is acceptable since these two functions can be computed directly by the algebraic formulas \eqref{2.11} and \eqref{2.12}.
	\label{rem indirect data}
\end{Remark}

Finding a vector $U(\xx)$ satisfying equation \eqref{2.10} and constraints \eqref{2.11} and \eqref{2.12} is the main point in our numerical method to find the function $f(\xx)$. 
In fact, having $U(\xx) = (u_1(\xx), \dots, u_2(\xx), \dots, u_N(\xx))$ in hand, we can compute the function $u(\xx, t)$ via \eqref{2.1111}. The desired function $f(\xx)$ is given by $u(\xx, t = 0).$

	Due to the truncation step in \eqref{2.1111}, equation \eqref{2.10} is not exact. 
	We call it an {\it  approximate model}.  
	Solving it, together with the ``over-determined" boundary conditions \eqref{2.11} and \eqref{2.12}, for the Fourier coefficients $(u_n(\xx))_{n = 1}^N$ of $u(\xx, t)$, $\xx \in \Omega$, $t \in [0, T]$, might not be rigorous.
	In fact, proving the ``accuracy" of  \eqref{2.10} when $N \to \infty$ is extremely challenging
	and is out of the scope of this paper.
	However, we experience in many earlier works that the solution of \eqref{2.10}, \eqref{2.11} and \eqref{2.12} well approximates Fourier coefficients of the function $u(\xx, t)$, leading to good solutions of variety kinds of inverse problems, see \cite{KlibanovKolesov:ip2018, KlibanovLiZhang:ip2019, KlibanovNguyen:ip2019, NguyenLiKlibanov:IPI2019}.

\begin{Remark}[The choice of $N$]
On $\Omega = (-2, 2)^2 $, we arrange $81 \times 81 $ grid points $\{(x_i, y_j): 1\leq i, j \leq 81\}$. 
In Figure \ref{fig 1} displays the functions of $u(\xx, t)$ and its approximation $ \sum_{n = 1}^N u_n(\xx)\Psi_n(t)$ where $u(\xx, t)$ is the true solution of the forward problem and $u_n(\xx)$, $n = 1, \dots, N$, is computed using \eqref{2.2}.
This numerical experiment suggests us to take $N = 30$. 
It is worth mentioning that when $N \leq 25$, the numerical solutions are not satisfactory, when $N = 30$, numerical results are quite accurate regardless the high noise levels and when $N \geq 35,$ the computation is time-consuming.
\begin{figure}[h!]
\begin{center}
	\subfloat[$N = 10$]{
			\includegraphics[width=.3\textwidth]{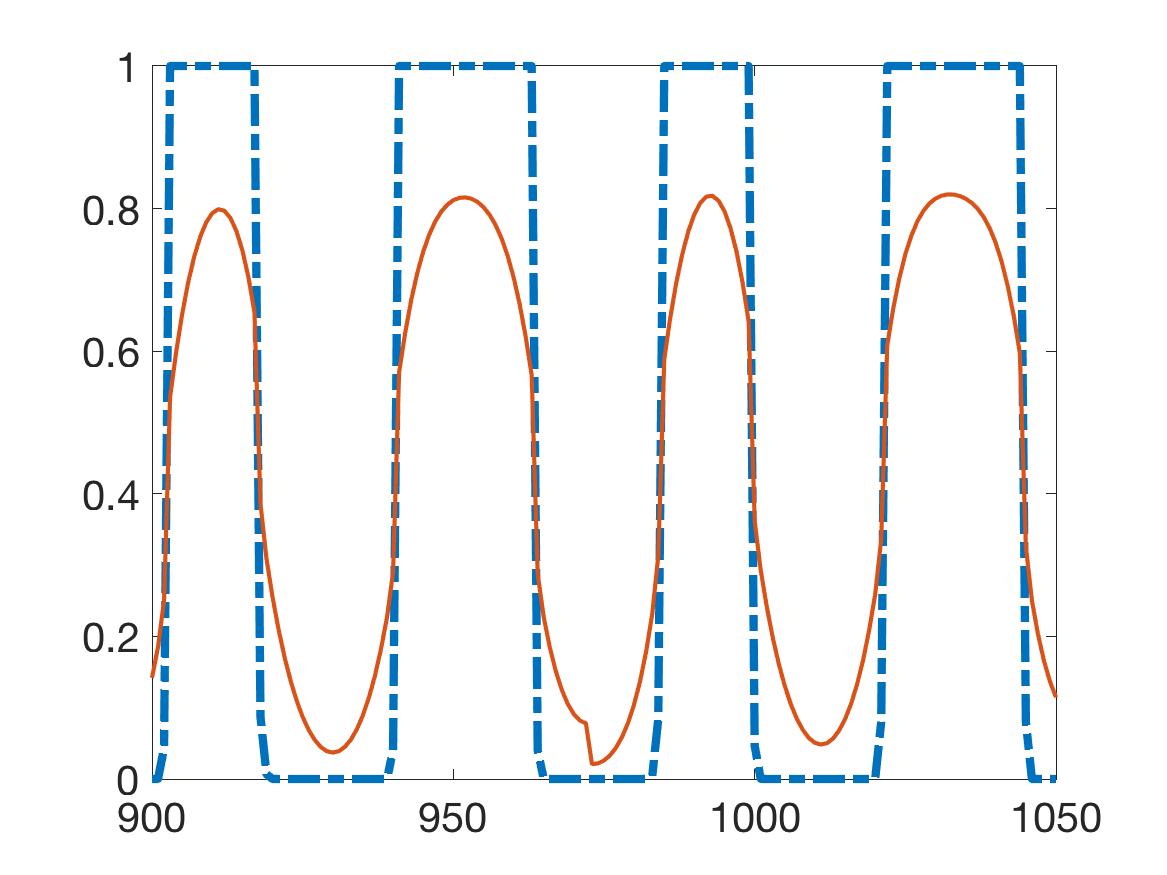}
		}
	\subfloat[$N = 20$]
	{
			\includegraphics[width=.3\textwidth]{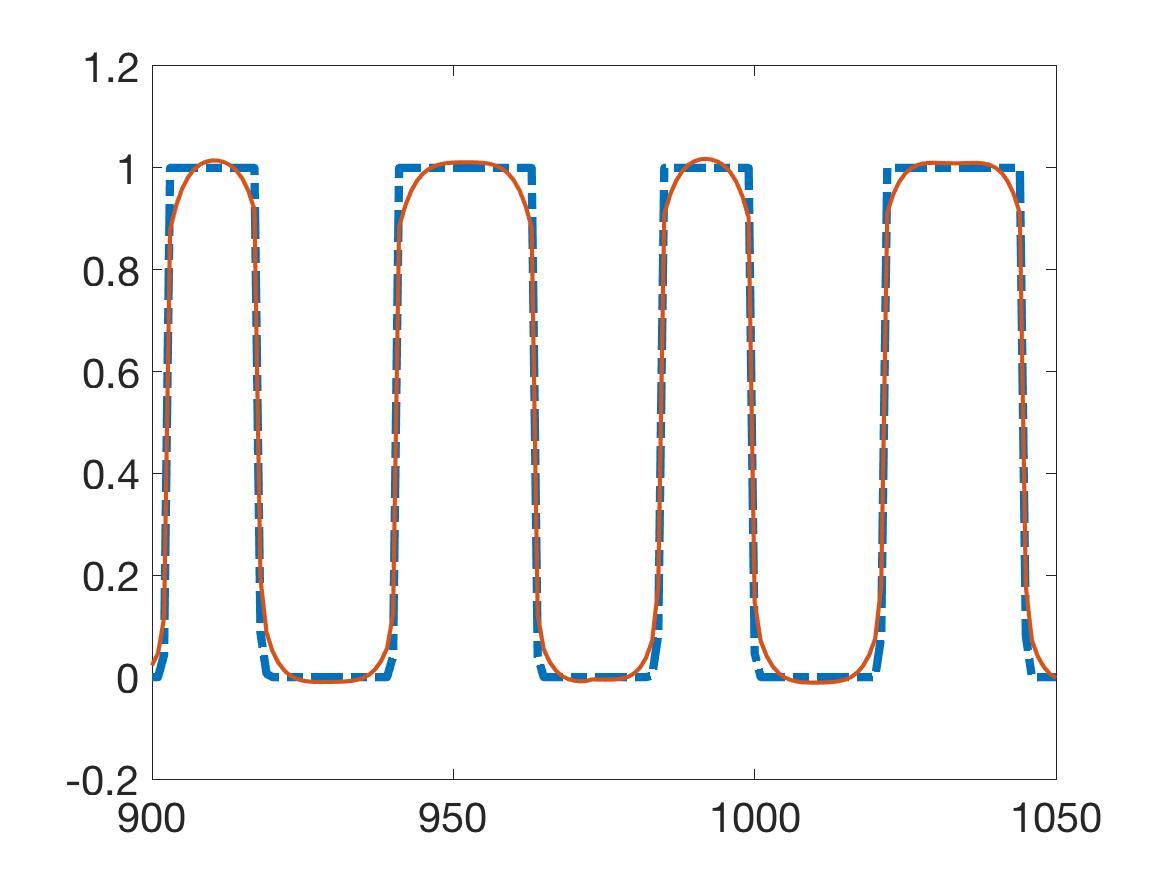}
		}	
	\subfloat[$N = 30$]
	{
			\includegraphics[width=.3\textwidth]{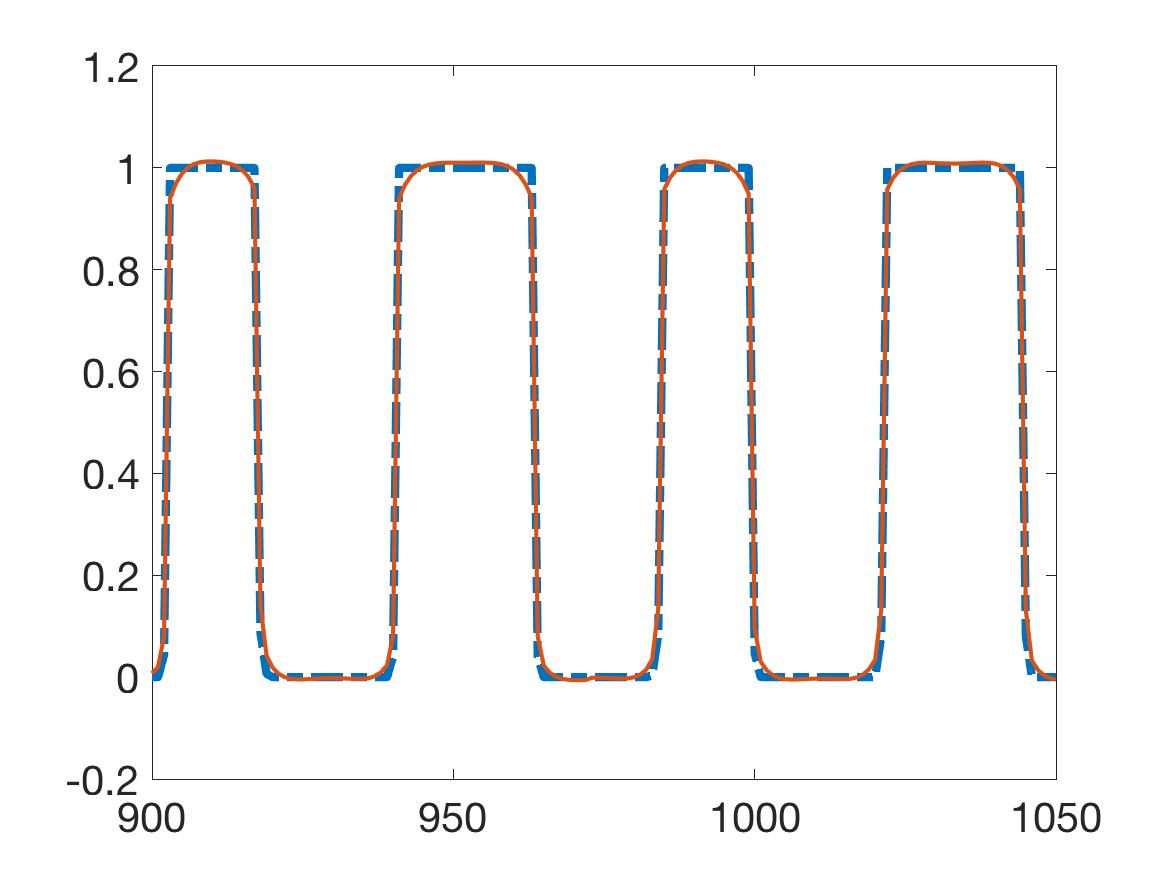}
		}
	\caption{\label{fig 1}The function $u(\xx, t = 0)$ (dash-dot) and its approximation $\sum_{n = 1}^N u_n(\xx) \Psi_n(t = 0)$ (solid) at the points numbered from 900 to 1050. These functions are taken from Test 4 in Section \ref{sec test}. It is evident that the larger $N$, the better approximation for the function $u$ is obtained by the $N^{\rm th}$ partial sum of the Fourier series in \eqref{Fourier series}.}
\end{center}
\end{figure}
\label{rem N}
\end{Remark}

\subsection{The quasi-reversibility method}

As mentioned, our method to solve Problem \ref{ISP} is based on a numerical solver for \eqref{2.10},  \eqref{2.11} and \eqref{2.12}.
We do so by employing the quasi-reversibility method;
that means, we minimize the functional
\begin{equation}
	J_{\epsilon}(U) = \int_{\Omega} |\calL U(\xx) - S U(\xx)|^2 d\xx + \epsilon \|U\|_{H^3(\Omega)}^2.
	\label{def J}
\end{equation}
subject to the constraints \eqref{2.11} and \eqref{2.12}.
Here $\epsilon$ is a positive number serving as a regularization parameter.
Impose the condition that the set of admissible data 
\begin{equation}
	H = \{V \in H^3(\Omega)^N: V|_{\partial \Omega \times [0, T]} = \widetilde F \mbox{ and } A \nabla V \cdot \nu|_{\partial \Omega \times [0, T]} = \widetilde G\} 
	\label{2.14}
\end{equation}
is nonempty, where $\widetilde F$ and $\widetilde G$ are our indirect data, see Remark \ref{rem indirect data}, defined in \eqref{2.11} and \eqref{2.12}.
The result below guarantees the existence and uniqueness for the minimizer of $J_{\epsilon}$, $\epsilon > 0$. 
\begin{proposition}
	Assume that the set of admissible data $H$, defined in \eqref{2.14}, is nonempty. Then, for all $\epsilon > 0$, the functional $J_{\epsilon}$ admits a unique minimizer satisfying \eqref{2.11} and \eqref{2.12}.
	This minimizer is called the regularized solution to \eqref{2.10}, \eqref{2.11} and \eqref{2.12}.
	\label{thm regularize}
\end{proposition}
\begin{proof}
	Proposition \ref{thm regularize} is an analog of \cite[Theorem 3.1]{NguyenLiKlibanov:IPI2019} whose proof is based on the Riesz representation theorem. 
	An alternative method to prove this proposition is
	 from the standard argument in convex analysis, see e.g. \cite{KlibanovLoc:ipi2016, LocNguyen:ip2019}.
\end{proof}	

The minimizer of $J_{\epsilon}$ in $H$ is called the {\it regularized solution} of \eqref{2.10}, \eqref{2.11} and \eqref{2.12} obtained by the quasi-reversibility method.

The analysis above leads to Algorithm \ref{alg1}, which describes our numerical method to reconstruct the function $f(\xx)$, $\xx \in \Omega$.
In the next section, we establish a new Carleman estimate. 
This estimate plays an important role in proving the convergence of the regularized solution, due to the quasi-reversibility method, to the true solution of \eqref{2.10}, \eqref{2.11} and \eqref{2.12} in Section \ref{sec convergence} as the measurement noise and $\epsilon$ tend to $0$.
\begin{algorithm}[h!]
\caption{\label{alg1}The procedure to solve Problem \ref{ISP}}
	\begin{algorithmic}[1]
	\State\, Choose a number $N$. Construct the functions $\Psi_m$, $1 \leq m \leq N,$ in Section \ref{sec basis} and compute the matrix $S$ whose the $mn^{\rm th}$ entry is given in \ref{matrix S}.
	\State\,  Calculate the boundary data $\widetilde F$ and $\widetilde G$ for the vector valued function $U$ via \eqref{2.11} and \eqref{2.12} respectively.
	\State\, \label{step quasi} Solve  \eqref{2.10}, \eqref{2.11}) and \eqref{2.12} via the quasi-reversibility method for the vector \[U(\xx) = (u_1(\xx), \dots, u_N(\xx))^T \quad \xx \in \Omega.\]
	\State\, Compute $u(\xx, t)$, $(\xx, t) \in \Omega \times [0, T]$ using \eqref{Fourier series}.
	\State\, Set the desired function $f(\xx) = u(\xx, 0)$ for all $\xx \in \Omega.$
	\end{algorithmic}
\end{algorithm}

\section{A Carleman estimate for second order elliptic operators on general domains} \label{sec Carleman}

Let the matrix $A$ be as in \eqref{A}. 
The main aim of this section is to prove a Carleman estimate in a general domain $\Omega$. 
Similar versions of Carleman estimate can be found in \cite[Theorem 3.1]{KlibanovLiZhang:ip2019} and \cite[Lemma 5]{MinhLoc:tams2015} when $\Omega$ is an annulus 
and \cite[Theorem 4.1]{NguyenLiKlibanov:IPI2019} and  when $\Omega$ is  a cube.
In this paper, we will use the following estimate to derive the convergence of the quasi-reversibility method. It can be deduced from \cite[Lemma 3, Chapter 4, \S 1]{Lavrentiev:AMS1986}.

Without lost of generality, we can assume that
\begin{equation}
	\Omega \subset 
	\left\{
		\xx = (x_1, x_2, \dots, x_d): 0 < x_1 + X^{-2} \sum_{i = 2}^d x_i^2 < 1
	\right\}
	\label{Omega bounded}
\end{equation}
for some $0 < X < 1$.
Define the function
\begin{equation}
	\psi(\xx) = x_1 + \frac{1}{2 X^2} \sum_{i = 1}^d x_i^2 + \alpha, \quad 0 < \alpha < 1/2.
	\label{psi def}
\end{equation}
Using Lemma 3 in \cite[Chapter 4, \S 1]{Lavrentiev:AMS1986} for the function $u \in C^2(\overline \Omega)$ that is independent of the time variable, we can find a constant $\sigma_0$ and a constant $\sigma_1$ (depending only on $\alpha$ and the entries $a_{ij}$, $1 \leq i, j \leq d$, of the matrix $A$) such that for all $\lambda \geq \sigma_0$ and $p > \sigma_1$
\begin{multline}
	\frac{\lambda p}{X^2} e^{2\lambda \psi^{-p}(\xx)} |\nabla u|^2 + \lambda^3 p^4 \psi^{-2p - 2}e^{2\lambda \psi^{-p}(\xx)} |u|^2
	\leq 
	-\frac{C \lambda p}{X^2} e^{2\lambda \psi^{-p}(\xx)}
	u\Div(A\nabla u)
	\\
	+ C \psi^{p + 2} e^{2\lambda \psi^{-p}(\xx)} |\Div (A\nabla u)|^2 + \Div U
	\label{LRS}
\end{multline}
for all $\xx \in \Omega$ where the vector $U$ satisfies 
\begin{equation}
	|U| \leq C e^{2\lambda \psi^{-p}(\xx)} \left(
		\frac{\lambda p}{X} |\nabla u|^2 + \frac{\lambda^3 p^3}{X^3} \psi^{-2p - 2} u^2
	\right).
	\label{LRS1}
\end{equation}
Applying \eqref{LRS} and \eqref{LRS1}, we have the lemma.

\begin{lemma}[Carleman estimate]
	Let $u \in C^2(\overline \Omega)$ satisfying \begin{equation}
		u|_{\partial \Omega} = A\nabla u \cdot \nu = 0 \quad \mbox{on } \partial \Omega
	\label{homo boundary}
	\end{equation}
	where $\nu$ the outward unit normal vector of $\partial \Omega.$
	Then, there exist a positive number $\sigma_0$ and $\sigma_1$, depending only on $\alpha$ and $A$, such that
	\begin{multline}
		\frac{\lambda p}{X^2}\int_{\Omega} e^{2\lambda \psi^{-p}(\xx)} |\nabla u|^2 d\xx 
		+ \lambda^3 p^4\int_{\Omega} \psi^{-2p - 2}e^{2\lambda \psi^{-p}(\xx)} |u|^2 d\xx
		\\
		\leq 	 C\int_{\Omega}  \psi^{p + 2} e^{2\lambda \psi^{-p}(\xx)} |\Div (A\nabla u)|^2 d\xx
		\label{Car es1}
	\end{multline}
	for $\lambda > \sigma_0$ and $p > \sigma_1$.
	In particular, fixing $p > \sigma_1$, one can find $\lambda > \sigma_0$ such that
	 \begin{multline}
		\lambda p\int_{\Omega} e^{2\lambda \psi^{-p}(\xx)} |\nabla u|^2 d\xx 
		+ \lambda^3 p^4\int_{\Omega} e^{2\lambda \psi^{-p}(\xx)} |u|^2 d\xx
		\\
		\leq 	 C\int_{\Omega}   e^{2\lambda \psi^{-p}(\xx)} |\Div (A\nabla u)|^2 d\xx.
		\label{Car es}
	\end{multline}
	\label{thm Car}
	\end{lemma}
\begin{proof}
	We claim that 
	\begin{equation}
		\nabla u(\xx) = 0 
		\quad 
		\mbox{for all } \xx \in \partial \Omega. 
		\label{no boundary}
	\end{equation}
	In fact, assume that $\nabla u(\xx) \not = 0$ at some points $\xx \in \partial\Omega.$ 
	Since $u(\xx) = 0$ on $\partial \Omega$, see \eqref{homo boundary},  $\nabla u(\xx) \cdot \tau(\xx) = 0$ where $\tau(\xx)$ is any tangent vector to $\partial \Omega$ at the point $\xx$. 
	Thus, $\nabla u(\xx)$ is perpendicular to $\partial \Omega$ at $\xx$. In other words, $  \nabla u(\xx) = \theta \nu(\xx)$ for some nonzero scalar $\theta$.
	 We have $0 = A(\xx) \nabla u(\xx) \cdot \nu(\xx) = \theta A(\xx)  \nu(\xx) \cdot \nu(\xx)$, which is a contradiction  to \eqref{coercive}.
	 
	Integrating both sides of \eqref{LRS}, we have
	\begin{multline}
		\frac{\lambda p}{X^2}\int_{\Omega} e^{2\lambda \psi^{-p}(\xx)} |\nabla u|^2 d\xx 
		+ \lambda^3 p^4\int_{\Omega} \psi^{-2p - 2}e^{2\lambda \psi^{-p}(\xx)} |u|^2 d\xx
		\\
		\leq  -\frac{C \lambda p}{X^2}\int_{\Omega} e^{2\lambda \psi^{-p}(\xx)}
	u\Div(A\nabla u) d\xx
	+ C\int_{\Omega}  \psi^{p + 2} e^{2\lambda \psi^{-p}(\xx)} |\Div (A\nabla u)|^2 d\xx.
	\label{3.5}
	\end{multline}
	Here, the term $\ds \int_{\Omega} \Div U d\xx$ is dropped because it vanishes due the the divergence theorem, \eqref{homo boundary} and \eqref{no boundary}
	Using the inequality $|ab| \leq \lambda p a^2 + \frac{1}{2\lambda p} b^2$
	\begin{multline}
		-\frac{C \lambda p}{X^2}\int_{\Omega} e^{2\lambda \psi^{-p}(\xx)}
	u\Div(A\nabla u) d\xx
	\\
	\leq 
	\frac{C \lambda^2 p^2}{X^2} \int_{\Omega} e^{2\lambda \psi^{-p}(\xx)} u^2 d\xx
	+
	\frac{C}{ X^2} \int_{\Omega} e^{2\lambda \psi^{-p}(\xx)} |\Div (A\nabla u)|^2d\xx.
	\label{3.6}
	\end{multline}
	Combining \eqref{3.5} and \eqref{3.6},
	we obtain
	\begin{multline*}
		\frac{\lambda p}{X^2}\int_{\Omega} e^{2\lambda \psi^{-p}(\xx)} |\nabla u|^2 d\xx 
		+ \lambda^3 p^4\int_{\Omega} \psi^{-2p - 2}e^{2\lambda \psi^{-p}(\xx)} |u|^2 d\xx
		\\
		\leq  
	\frac{C \lambda^2 p^2}{X^2} \int_{\Omega} e^{2\lambda \psi^{-p}(\xx)} u^2 d\xx
	+
	\frac{C}{ X^2} \int_{\Omega} e^{2\lambda \psi^{-p}(\xx)} |\Div (A\nabla u)|^2d\xx	
	\\
	+ C\int_{\Omega}  \psi^{p + 2} e^{2\lambda \psi^{-p}(\xx)} |\Div (A\nabla u)|^2 d\xx.
	\end{multline*}
	Fixing $p >  \sigma_1$ and choosing $\lambda$ large such that the second term in the left hand side dominates the first term in the right hand side, we obtain	
	\begin{multline*}
		\frac{\lambda p}{X^2}\int_{\Omega} e^{2\lambda \psi^{-p}(\xx)} |\nabla u|^2 d\xx 
		+ \lambda^3 p^4\int_{\Omega} \psi^{-2p - 2}e^{2\lambda \psi^{-p}(\xx)} |u|^2 d\xx
		\\
		\leq 	 C\int_{\Omega}  \psi^{p + 2} e^{2\lambda \psi^{-p}(\xx)} |\Div (A\nabla u)|^2 d\xx.
	\end{multline*}
	The estimate \eqref{Car es1} follows.
\end{proof}

%
%

\section{The convergence of the quasi-reversibility method}
\label{sec convergence}

In this section, we continue to assume \eqref{Omega bounded}.
Let $\widetilde F^*$ and $\widetilde G^*$ be the noiseless data for \eqref{2.11} and \eqref{2.12}, see Remark \eqref{rem indirect data}, respectively. 
The noisy data are denoted by $\widetilde F^{\delta}$ and $\widetilde G^{\delta}$. 
Here $\delta$ is the noise level. 
In this section, assume that there exists $\calE \in H^3(\Omega)^N$ such that 
\begin{enumerate}
\item for all $\xx \in \Omega$ 
\begin{equation}
	\calE(\xx) = \widetilde F^{\delta}(\xx) - \widetilde F^*(\xx) \quad \mbox{and } \quad A(\xx) \nabla\calE(\xx) \cdot \nu(\xx) = \widetilde G^{\delta}(\xx) - \widetilde G^*(\xx);
	\label{error fns}
\end{equation}
\item and the bound
\begin{equation}
	\|\calE\|_{H^3(\Omega)^N}  < \delta \quad \mbox{as } \delta \to 0^+
	\label{noise level}
\end{equation}
holds true.
\end{enumerate}
The assumption about the existence of $\calE$ satisfying \eqref{error fns} and \eqref{noise level} is equivalent to the condition
\[
	\inf\{\|\Phi\|_{H^3(\Omega)^N}: \Phi|_{\partial \Omega} =  \widetilde F^{\delta} - \widetilde F^*, \partial_{\nu} \Phi|_{\partial \Omega} = \widetilde G^{\delta} - \widetilde G^*\} < \delta.
\] 
In this section,
we establish the following result to study the accuracy of the quasi-reversibility method.
\begin{theorem}
	Assume that $U^*$ is the function that satisfies \eqref{2.10}, \eqref{2.11} and \eqref{2.12} with $\widetilde F$ and $\widetilde G$ replaced by $\widetilde F^*$ and $\widetilde G^*$ respectively. 
	Fix $\epsilon > 0.$ Let $U^{\delta}$ be the minimizer of $J_{\epsilon}$ subject to constraints \eqref{2.11} and \eqref{2.12} with $\widetilde F$ and $\widetilde G$ replaced by $\widetilde F^\delta$ and $\widetilde G^\delta$ respectively. 
	Assume further that there is an ``error" function $\calE$  in $H^3(\Omega)^N$ satisfying \eqref{error fns} and \eqref{noise level}.
	Then, we have the estimate
	\begin{equation}
		\|U^\delta - U^*\|_{H^1(\Omega)^N}^2 \leq C\left(\delta^2 + \epsilon \|U^*\|_{H^3(\Omega)^N}^2\right)
	\label{convergence QRM}
	\end{equation}
	where $C$ is a constant that depends only on $\Omega,$ $\|A\|_{C^1(\overline \Omega)}$ and $\mu$.
\end{theorem}

\begin{proof}
	Since $U^{\delta}$ is the minimizer of $J_{\epsilon}$, by the variational principle, we have
	\begin{equation}
		\langle \calL U^\delta - S U^{\delta}, \calL \Phi - S \Phi\rangle_{L^2(\Omega)^N} 
		+ 
		\epsilon \langle U^{\delta}, \Phi \rangle_{H^3(\Omega)^N} = 0
		\label{4.3}
	\end{equation}
	for all test functions $\Phi$ in the space
	\[
		H_0^3(\Omega)^N = \{
			\phi \in H^3(\Omega)^N: \phi = A\nabla \phi \cdot \nu = 0 \mbox{ on } \partial \Omega
		\}.
	\]
	Since $\calL U^* - S U^* = 0,$ we can deduce from \eqref{4.3} that
	\[
		\langle \calL (U^\delta - U^*) - S (U^{\delta} - U^*), \calL \Phi - S \Phi\rangle_{L^2(\Omega)^N} 
		+ 
		\epsilon \langle U^{\delta} - U^*, \Phi \rangle_{H^3(\Omega)^N} = -\epsilon \langle  U^*, \Phi \rangle_{H^3(\Omega)^N}.
	\]
	Plugging the test function 
	\begin{equation}
		\Phi = U^{\delta} - U^* - \calE \in H^3_0(\Omega)
		\label{phi}
	\end{equation}
	 into the identity above, we have
	\begin{multline*}
		\|\calL \Phi - S\Phi\|_{L^2(\Omega)^N}^2
		+\epsilon \|\Phi\|_{H^3(\Omega)^N}^2
		= -\langle \calL \calE - S\calE, \calL \Phi - S\Phi \rangle_{L^2(\Omega)^N}^2 
		\\
		-\epsilon \langle \calE, \Phi\rangle_{H^3(\Omega)^N} 
		-\epsilon \langle  U^*, \Phi \rangle_{H^3(\Omega)^N}.
	\end{multline*}	
	Applying the  Cauchy--Schwartz inequality and removing lower order terms, we obtain
	\begin{equation}
		\|\calL \Phi - S \Phi\|_{L^2(\Omega)^N}^2 
		+ \epsilon \|\Phi\|_{H^3(\Omega)^N}^2 \leq C\left(\delta^2 + \epsilon \|U^*\|_{H^3(\Omega)^N}^2\right). 
	\label{4.4}
	\end{equation}
	Recall from \eqref{calL} that
	\[
		\|\calL \Phi - S \Phi\|_{L^2(\Omega)^N}^2 
		= \|\Div (A(\xx) \nabla \Phi + \bb(\xx) \cdot \nabla \Phi + (c(\xx)\Id - S)\Phi\|^2_{L^2(\Omega)^N}
	\]
	Recall the function $\psi$ in \eqref{psi def}. 
	Fix $\lambda > \sigma_0$ and $p > \sigma_1$ as in Lemma \ref{thm Car}. 
	Set 
	\[
		M = \max\{e^{2\lambda \psi^{-p}(
		\xx)}: \xx \in \overline\Omega\}
		\quad
		\mbox{and }
		\quad
		m = \min\{e^{2\lambda \psi^{-p}(
		\xx)}: \xx \in \overline\Omega\}.
	\] 
	We have
	\begin{multline*}
		M\|\calL \Phi - S \Phi\|_{L^2(\Omega)^N}^2 
		\geq
		\frac{1}{2}\big\|e^{\lambda \psi^{-p}(\xx)}\Div (A(\xx) \nabla \Phi)
		\\
		+ e^{\lambda \psi^{-p}(\xx)}(\bb(\xx) \cdot \nabla \Phi + (c(\xx)\Id - S)\Phi)
		\big\|^2_{L^2(\Omega)^N}.
	\end{multline*}
	Using the inequality $(x - y)^2 \geq \frac{1}{2} x^2 - y^2$, we have
	\begin{multline*}
		M\|\calL \Phi - S \Phi\|_{L^2(\Omega)^N}^2 
		\geq
		\frac{1}{2}\|e^{\lambda \psi^{-p}(\xx)}\Div (A(\xx) \nabla \Phi)\|_{L^2(\Omega)^N}^2 
		\\
		- \|e^{\lambda \psi^{-p}(\xx)}(\bb(\xx) \cdot \nabla \Phi + (c(\xx)\Id - S)\Phi)\|^2_{L^2(\Omega)^N}.
	\end{multline*}
Hence,
	Thus, by \eqref{Car es},
	\begin{multline}
		M\|\calL \Phi - S \Phi\|_{L^2(\Omega)^N}^2 
		\geq
		C\lambda p\int_{\Omega} e^{2\lambda \psi^{-p}(\xx)} |\nabla \Phi|^2 d\xx 
		+ C\lambda^3 p^4\int_{\Omega} e^{2\lambda \psi^{-p}(\xx)} |\Phi|^2 d\xx
		\\
		- \|e^{\lambda \psi^{-p}(\xx)}(\bb(\xx) \cdot \nabla \Phi + (c(\xx)\Id - S)\Phi)\|^2_{L^2(\Omega)^N}.
		\label{39}
	\end{multline}
	Now, fixing $\lambda > \sigma_0$ large, we obtain from \eqref{39} that
	\begin{equation}
		M|\calL \Phi - S \Phi\|_{L^2(\Omega)^N}^2 \geq C m\|\Phi\|_{H^1(\Omega)^N}.
		\label{4.5}
	\end{equation}
	Here, we have used the boundedness of $\bb$ and $c$ in $\Omega.$
	Combining \eqref{phi}, \eqref{4.4} and \eqref{4.5} gives
	\[
		\|U^\delta - U^* - \calE\|^2_{H^1(\Omega)^N} \leq C \left(\delta^2 + \epsilon \|U^*\|_{H^3(\Omega)^N}^2\right).
	\]
	This and the assumption $\|\calE\|_{H^3(\Omega)^N} \leq C\delta$ imply inequality \eqref{convergence QRM}.
\end{proof}	

\begin{corollary}
	Let $f^*(\xx) = u^*(\xx, 0)$ and $f^{\delta}(\xx) = u^\delta(\xx, 0)$ where $u^*(\xx, t)$ and $u^{\delta}(\xx, t)$ are computed from $U^*(\xx)$ and $U^{\delta}(\xx)$ via \eqref{Fourier series} and \eqref{capU}.
	Then, by the trace theory
	\[
		\|f^{\delta} - f^*\|_{L^2(\Omega)} \leq C\left(\delta + \sqrt{\epsilon}\|U^*\|_{H^3(\Omega)^N}\right).
	\]
\end{corollary}
\section{Numerical illustrations} \label{sec num}

We numerically test our method when $d = 2$.
The domain $\Omega$ is the square $(-R, R)^2$. In this section, we write $\xx = (x, y)$. 
For the coefficients of the governing equation, we choose, for simplicity,
$A(\xx) = \Id$ and $\bb(\xx) = 0.$ 
The function $c$ is set as
\[
	c(x, y) =  (3(1-x)^2  e^{-x^2 - (y+1)^2}
   - 10(x/5 - x^3 - y^5)e^{-x^2-y^2}
   - 1/3e^{-(x+1)^2 - y^2})/10
\]
which is a scale of the ``peaks" function in Matlab. The graph of $c$ is displayed in Figure \ref{fig c}.
\begin{figure}[h!]
	\begin{center}
		\includegraphics[width = .3\textwidth]{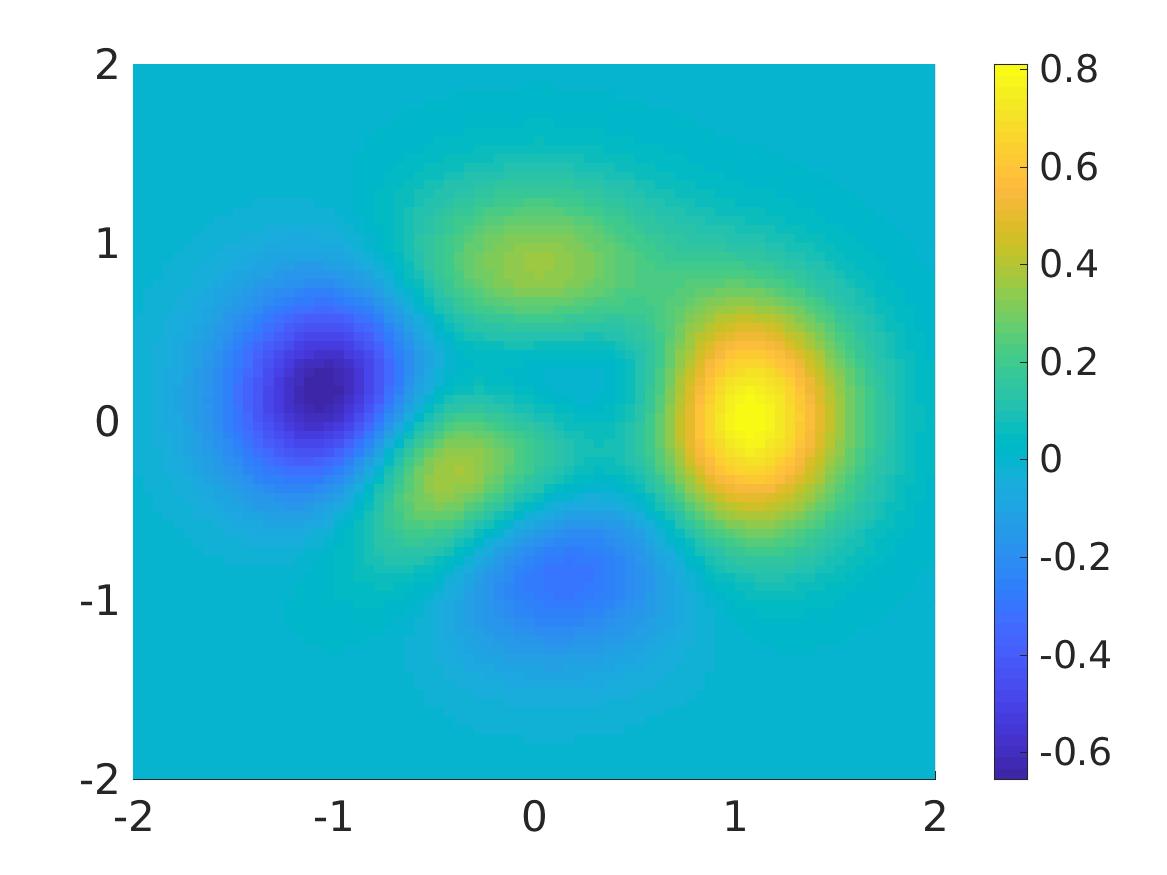}
	\end{center}
	\caption{\label{fig c} The true function $c(\xx)$ used for all numerical examples in this section.}
\end{figure}

Define a grid of points in $\Omega$
\[
	\mathcal{G} = 
	\{
		(x_i, y_j) = (-R + (i - 1)d_\xx, -R + (j - 1)d\xx): 1\leq i, j \leq N_\xx + 1
	\}
\]
where $N_\xx = 80$ and $d_\xx = 2R/N_\xx.$
For the time variable, we choose $T = 4$. 
Define a uniform partition of $[0, T]$ as
\[
	0 = t_1 < t_2 < \dots < t_{N_T + 1} = T
\] with step size $d_t = T/N_T$. In our tests, $N_T = 250.$
The forward problem is solved by finite difference method in the implicit scheme.
Denote by $u^*$ the solution of the forward problem. The data is given by
\[
	F(\xx, t) = u(\xx, t)(1 + \delta(2\rand - 1)) \quad 
	G(\xx, t) = \partial_{\nu} u(\xx, t)(1 + \delta(2\rand - 1))
\]
for $(\xx, t) \in \partial \Omega \times [0, T]$ where $\rand$ is the uniformly distributed random function taking value in $[0, 1]$ and $\delta$ is the noise level. The noise level $\delta$ is given in each numerical tests.

\subsection{The implementation for Algorithm \ref{alg1}}
\label{sec impl}

The main part of this section is to compute the minimizer $U$ of $J_{\epsilon}$ subject to the constraints \eqref{2.11} and \eqref{2.12}.
The ``cut-off" number $N$ is set to be 30, see Remark \ref{rem N} for this choice of $N$.
To construct the orthonormal basis $\{\Psi_n\}_{n = 1}^N,$ for each $n \in \{1, \dots, N\}$, we identify the function $\Phi_n$, defined in \eqref{phi}, by the $N_T + 1$ dimensional vector $(\Phi_n(t_1), \dots, \Phi_{n}(t_{N_T + 1}))$. 
Then apply the  Gram-Schmidt orthogonalization process for the set $\{(\Phi_n(t_1), \dots, \Phi_{n}(t_{N_T + 1}))\}_{n = 1}^N$ in the $N_T + 1$ dimensional Euclidian space. In other words, we construct $\{\Psi_n\}_{n = 1}^N$ in the finite difference scheme.
The discretized version of $U(\xx) = (u_1(\xx), \dots, u_N(\xx))^T$, $\xx \in \Omega$ is $(u_1(x_i, y_j), \dots, u_N(x_i, y_j))_{i, j = 1}^{N_\xx + 1}.$
Hence, $J_{\epsilon}(U)$, see \eqref{def J}, is approximated by
\begin{multline}
	J_{\epsilon}(U) =
	\\
	 d_\xx^2 \sum_{i, i = 2}^{N_\xx} 
	\sum_{m = 1}^N\Big|\frac{u_m(x_{i + 1}, y_j) + u_m(x_{i - 1}, y_j) + u_m(x_i, y_{j + 1}) + u_m(x_i, y_{j - 1}) - 4 u_m(x_i, y_j)}{d_\xx^2}
	\\
	+ c(x_i, y_j) u_m(x_i, y_j) - \sum_{n = 1}^N s_{mn} u_n(x_i, y_j)
\Big|^2
	+ \epsilon d_\xx^2\sum_{i, j = 2}^{N_x} \sum_{m = 1}^N|u_m(x_i, y_j)|^2
	\\
	+ \epsilon^2 d_\xx^2 \sum_{i, j = 2}^{N_\xx} \sum_{m = 1}^N\Big|
		\frac{u_m(x_{i+1}, y_j) - u_{m}(x_{i}, y_j) }{d_\xx}
	\Big|^2
	\\
	+ \epsilon^2 d_\xx^2 \sum_{i, j = 2}^{N_\xx} \sum_{m = 1}^N\Big|
		\frac{u_m(x_{i}, y_{j + 1}) - u_{m}(x_{i}, y_j) }{d_\xx}
	\Big|^2.
	\label{5.1}
\end{multline}
Here, we slightly change the $H^3$ norm of the regularity term to the $H^1$ norm. 
This makes the computational codes less heavy. 
The numerical results with this change are still acceptable.
We also modify the regularized parameter of the term $\|\nabla U\|_{L^2(\Omega)^N}$ to be $\epsilon^2$, instead of $\epsilon$, since we observe that the obtained numerical results are more accurate with this modification. 
To numerically prove this, we solve the inverse problem when the function $f_{\rm true}$ is given in Test 1 in Section \ref{sec test} in two cases: with and without this modification and then compare the corresponding outputs. 
The results are displayed in  Figure \ref{compare}. 
It is clear from Figure \ref{compare} that the modification above provides better numerical results.
\begin{figure}[h!]
	\begin{center}
		\subfloat[The function $f_{\rm true}$]{
			\includegraphics[width = .3\textwidth]{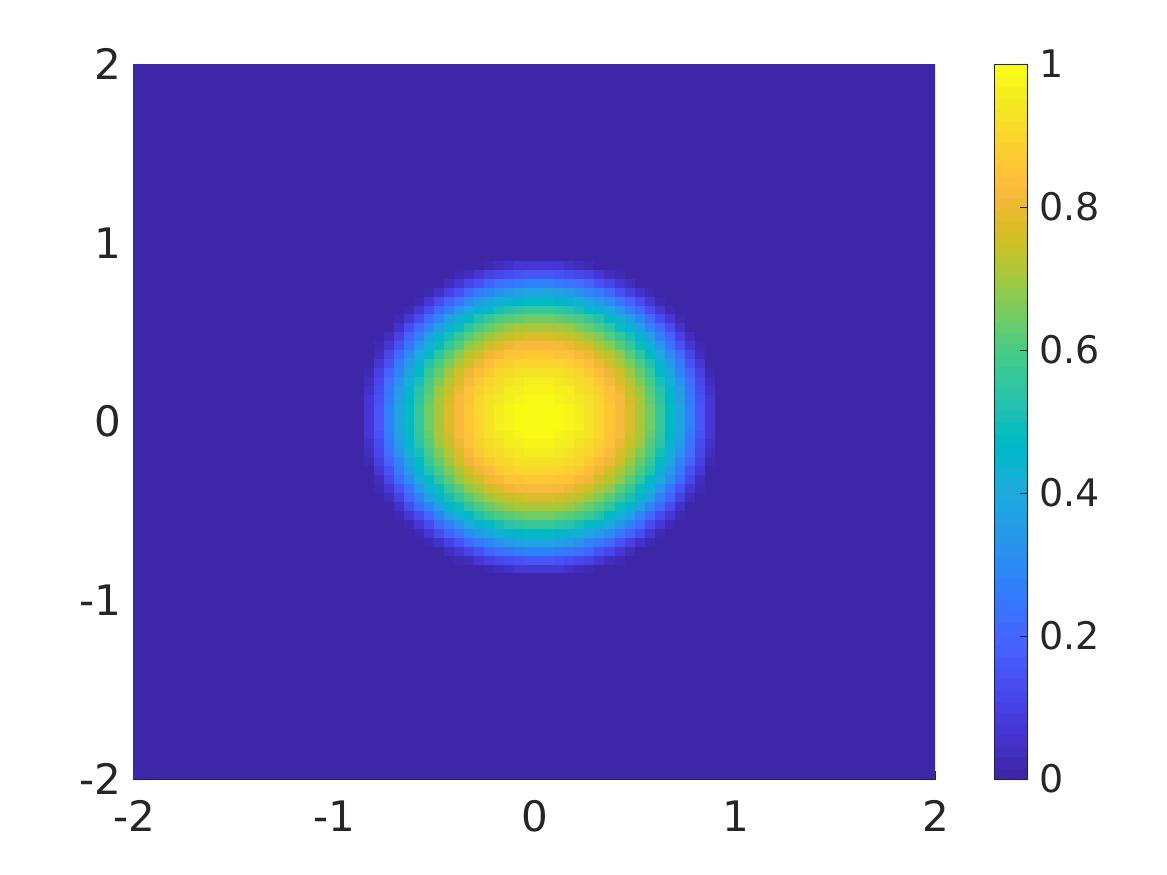} 
		} \hfill
		\subfloat[$f_{\rm comp}$ computed using the regularization term $\epsilon^2 \|\nabla U\|_{L^2(\Omega)^N}^2$ when $\delta = 50\%$]{
			\includegraphics[width = .3\textwidth]{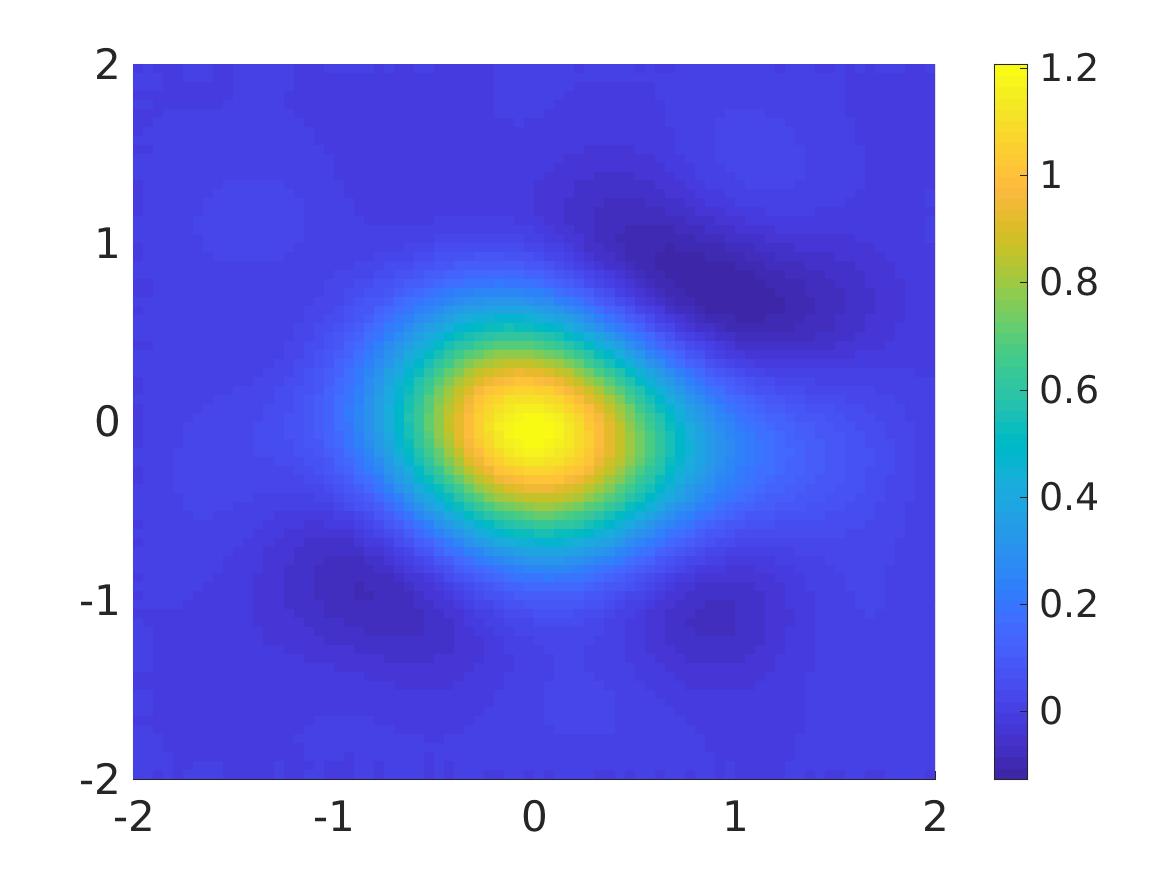} 
		}\hfill
		\subfloat[$f_{\rm comp}$ computed using the regularization term $\epsilon \|\nabla U\|_{L^2(\Omega)^N}^2$ when $\delta = 50\%$]{
			\includegraphics[width = .3\textwidth]{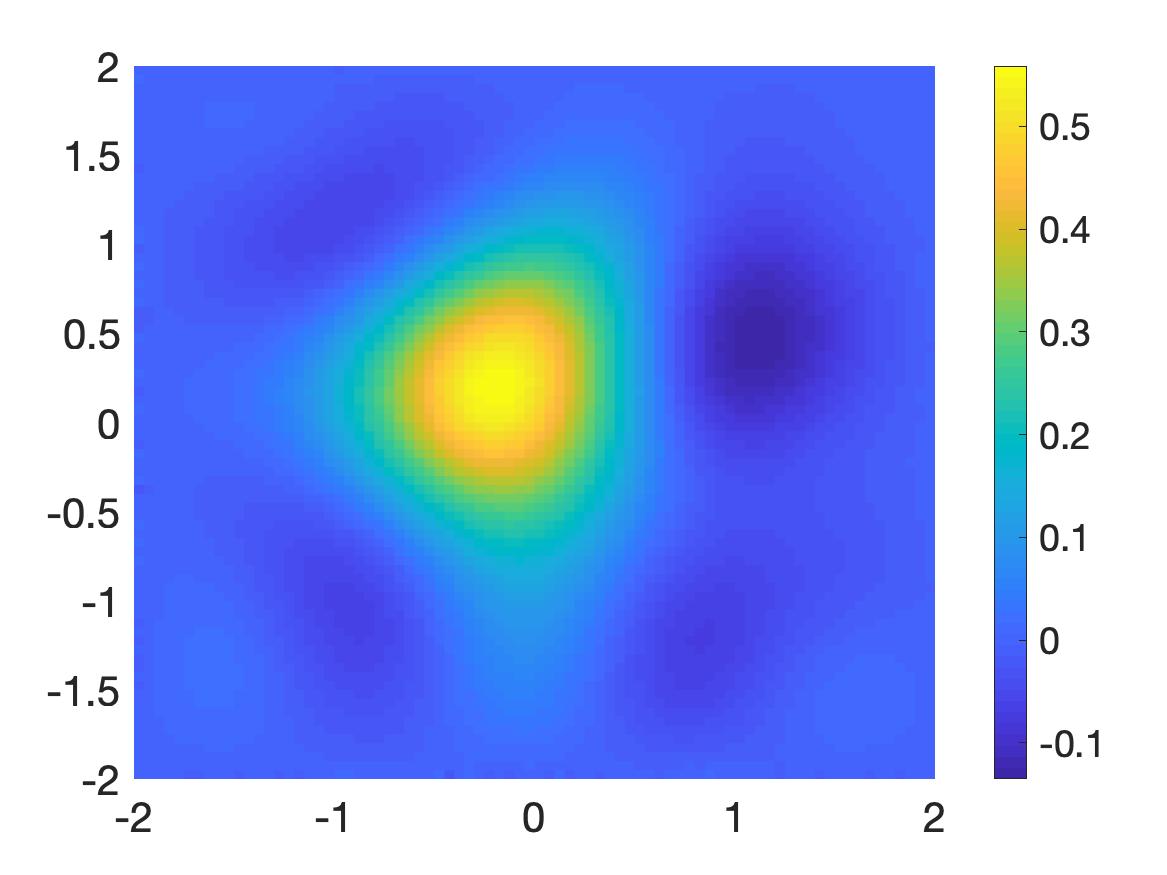} 
		}
		\caption{\label{compare} Test 1. The comparison of the reconstruction of the function $f$ with and without the modification for the regularized parameter.
		It is evident that the numerical result in (b) is significantly better than that in (c) in both reconstructed shape and computed value.}
	\end{center}
\end{figure}

The expression in \eqref{5.1} is simplified as follows
\begin{multline}
	J_{\epsilon}(U) = \\
	d_\xx^2 \sum_{i, j = 2}^{N_x} \sum_{m = 1}^N 
	\Big|
		\sum_{n = 1}^N \left[\delta_{mn}\left(\frac{-4}{d_\xx^2} + c(x_i, y_j)\right) -s_{mn} \right]u_n(x_i, y_j)
		+
		\frac{\delta_{mn}}{d_\xx^2}u_n(x_{i+1},y_j)
		\\
		+
		\frac{\delta_{mn}}{d_\xx^2}u_n(x_{i-1},y_j)
		+
		\frac{\delta_{mn}}{d_\xx^2}u_n(x_{i},y_{j+1})
		+
		\frac{\delta_{mn}}{d_\xx^2}u_n(x_{i},y_{j-1})
	\Big|^2
	\\
	+ \epsilon d_\xx^2\sum_{i, j = 2}^{N_x} \sum_{m = 1}^N|u_m(x_i, y_j)|^2
	+ \epsilon^2 d_\xx^2 \sum_{i, j = 2}^{N_\xx} \sum_{m = 1}^N\Big|
		\frac{u_m(x_{i+1}, y_j) - u_{m}(x_{i}, y_j) }{d_\xx}
	\Big|^2
	\\
	+ \epsilon^2 d_\xx^2 \sum_{i, j = 2}^{N_\xx} \sum_{m = 1}^N\Big|
		\frac{u_m(x_{i}, y_{j + 1}) - u_{m}(x_{i}, y_j) }{d_\xx}
	\Big|^2.
	\label{5.2}
\end{multline}
Here, we use the Kronecker number $\delta_{mn}$ for the convience of writing the computational codes.
We next identify $\{u_n(x_i, y_j): 1 \leq i, j \leq N_\xx + 1, 1 \leq n \leq N\}$ with the $(N_\xx + 1)^2N$ dimensional vector 
$
	\mathfrak{U} = (\mathfrak{u}_{\mathfrak{i}})_{\mathfrak{i} = 1}^{(N_\xx + 1)^2N}
$
according to the rule
$
	\mathfrak{u}_{\mathfrak{i}} = u_n(x_i, y_j)
$
where the index $\mathfrak{i}$ is 
\[
	\mathfrak{i} = (i - 1)(N_\xx + 1)N + (j - 1)N + n, \quad 1 \leq i, j \leq N_\xx + 1, 1 \leq n \leq N.
\]
Then, with this notation, $J_{\epsilon}(U)$ in \eqref{5.2} is rewritten as
\[
	\mathfrak{J}_{\epsilon}(\mathfrak{U}) = d_\xx^2 |\mathfrak{L} \mathfrak{U}|^2 + 
	+ \epsilon d_\xx^2 |\mathfrak{U}|^2
	+ \epsilon d_\xx^2 |D_x \mathfrak{U}|^2 
	+\epsilon d_\xx^2 |D_y \mathfrak{U}|^2.
\]
The $(N_\xx + 1)^2N \times (N_\xx + 1)^2N$ matrices $\mathfrak{L}$, $D_x$ and $D_y$ are as follows.
\begin{enumerate}
	\item {\it Define the matrix $\mathfrak{L}.$} For $\mathfrak{i} = (i - 1)(N_\xx + 1)N + (j - 1)N + m$, for some $2 \leq i, j \leq N_x$, the $\mathfrak{i} \mathfrak{j}^{\rm th}$ entry of $\mathfrak{L}$ is
	\begin{enumerate}
		\item  $\delta_{mn}\left(-4/d_\xx^2 + c(x_i, y_j)\right) -s_{mn} $ if $\mathfrak{j} = (i - 1)(N_\xx + 1)N + (j - 1)N + n,$
		\item $\delta_{mn}/d_x^2$ if  $\mathfrak{j} = (i \pm 1 - 1)(N_\xx + 1)N + (j - 1)N + n$ or $\mathfrak{j} = (i  - 1)(N_\xx + 1)N + (j \pm 1 - 1)N + n$,
		\item $0$ otherwise.
	\end{enumerate}
	\item {\it Define the matrix $D_x$.} For $\mathfrak{i} = (i - 1)(N_\xx + 1)N + (j - 1)N + m$, for some $2 \leq i, j \leq N_x$, the $\mathfrak{i} \mathfrak{j}^{\rm th}$ entry of $D_x$ is
	\begin{enumerate}
		\item $1/d_\xx$ if $\mathfrak{j} = (i + 1 - 1)(N_\xx + 1)N + (j - 1)N + m$,
		\item $-1/d_\xx$ if $\mathfrak{j} = \mathfrak{i}$,
		\item $0$ otherwise.
	\end{enumerate}
	\item {\it Define the matrix $D_x$.} For $\mathfrak{i} = (i - 1)(N_\xx + 1)N + (j - 1)N + m$, for some $2 \leq i, j \leq N_x$, the $\mathfrak{i} \mathfrak{j}^{\rm th}$ entry of $D_x$ is
	\begin{enumerate}
		\item $1/d_\xx$ if $\mathfrak{j} = (i  - 1)(N_\xx + 1)N + (j + 1 - 1)N + m$,
		\item $-1/d_\xx$ if $\mathfrak{j} = \mathfrak{i}$,
		\item $0$ otherwise.
	\end{enumerate}
\end{enumerate}

\begin{remark}[The values of the parameters]
	As mentioned, we take $N = 30$, $N_\xx = 80$, $N_T = 250$, $R = 2$. 
	The regularized parameter $\epsilon = 10^{-7}$.
	These values of parameters are used for all tests in Section \ref{sec test}.
\end{remark}

\subsection{Tests} \label{sec test}
We perform four (4) numerical examples in this paper. 
These examples with high levels of noise show the strength of our method. 
We will also compare the reconstructed maximum values of the reconstructed functions and the true ones.
Below, $f_{\rm true}$ and $f_{\rm comp}$ are, respectively, the true source function and the reconstructed one due to Algorithm \ref{alg1} with the parameters in Section \ref{sec impl}.

\begin{enumerate}
	\item {\it Test 1. The case of one inclusion.} The function $f_{\rm true}$ is a smooth function supported in a disk with radius 1 centered at the origin. 
	More precisely, 
	\[
		f_{\rm true}(\xx) = \left\{
			\begin{array}{ll}
				\ds e^{-\frac{1}{1 - |\xx|^2} + 1} &\mbox{if } |\xx| < 1,\\
				0 &\mbox{otherwise}.
			\end{array}
		\right.
	\]
	Figure \ref{test 1} displays the functions $f_{\rm true}$ and $f_{\rm comp}$. 
	Table \ref{table 1} show the reconstructed value of the function $f_{\rm comp}$ and the relative error.
	The noise levels are $\delta = 0\%$, $25\%$, $50\%$, $75\%$ and $100\%.$
	\begin{figure}[h!]
	\begin{center}
		\subfloat[The function $f_{\rm true}$]{
			\includegraphics[width = .3\textwidth]{ftrue_1} 
		} \hfill
		\subfloat[$f_{\rm comp}$, $\delta = 0\%$]{
			\includegraphics[width = .3\textwidth]{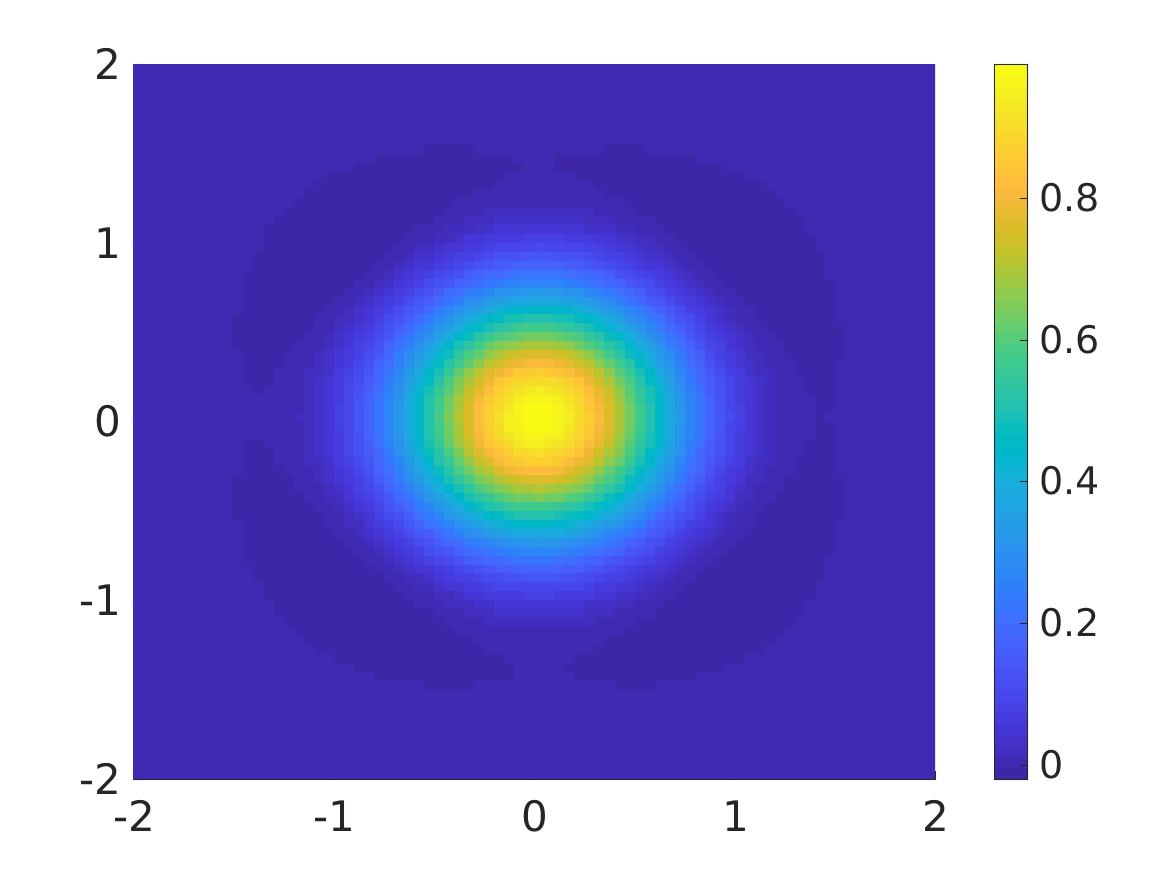} 
		}\hfill
		\subfloat[$f_{\rm comp}$, $\delta = 25\%$]{
			\includegraphics[width = .3\textwidth]{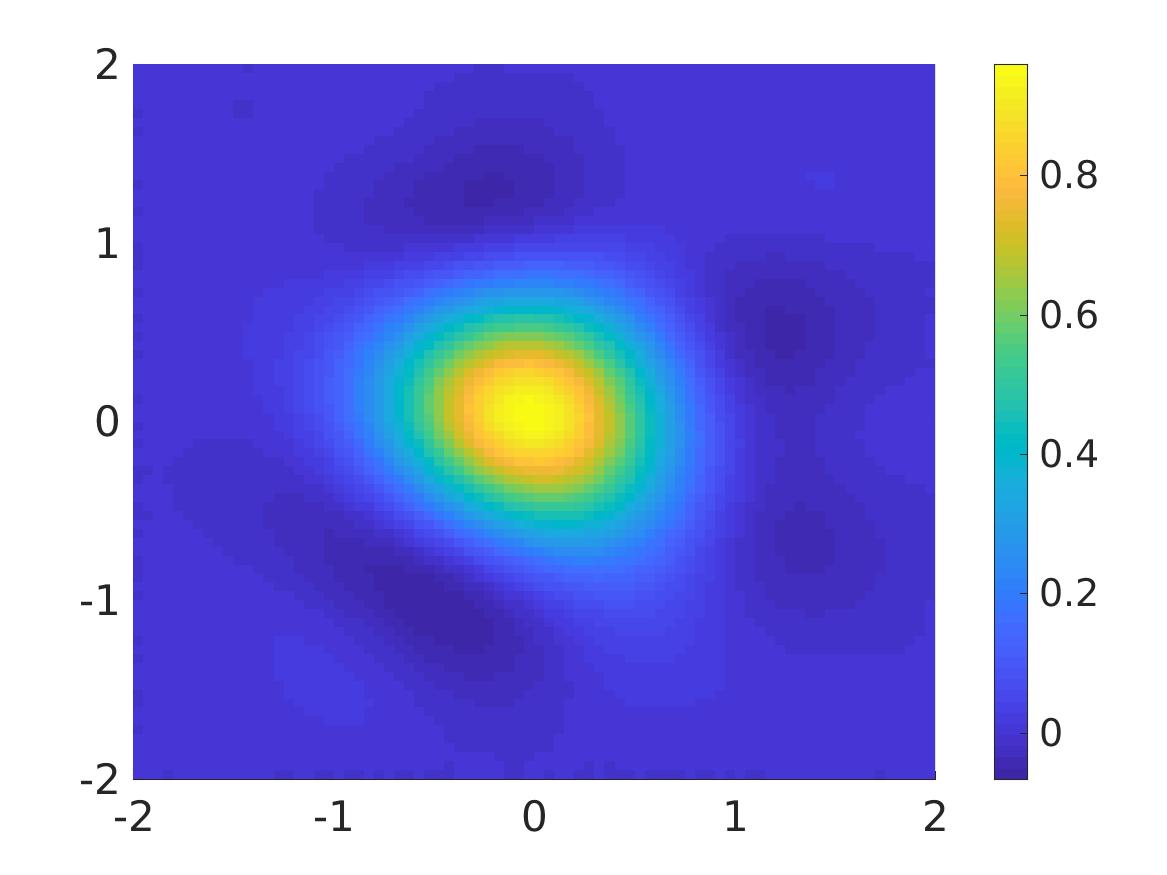} 
		}
		
		\subfloat[$f_{\rm comp}$, $\delta = 50\%$]{
			\includegraphics[width = .3\textwidth]{f_comp_1_50} 
		}\hfill
		\subfloat[$f_{\rm comp}$, $\delta = 75\%$]{
			\includegraphics[width = .3\textwidth]{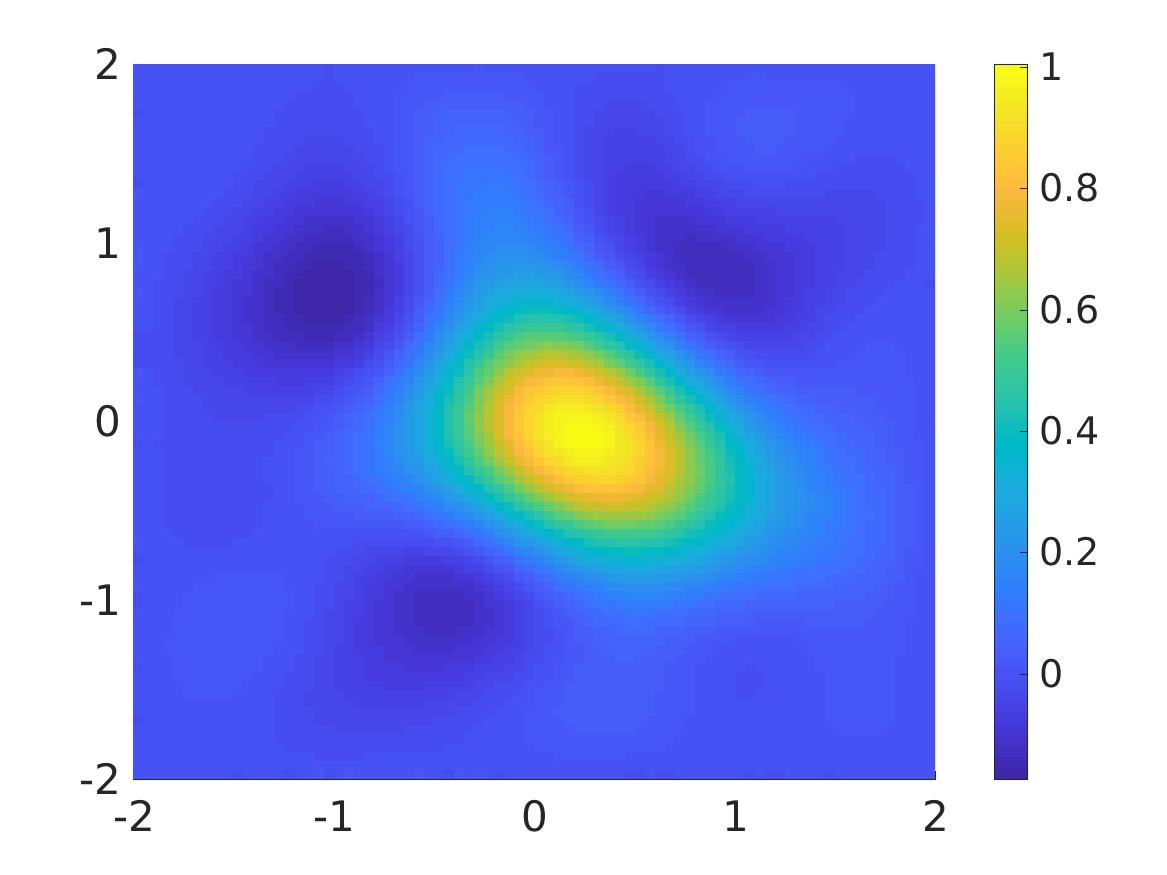} 
		}\hfill
		\subfloat[$f_{\rm comp}$, $\delta = 100\%$]{
			\includegraphics[width = .3\textwidth]{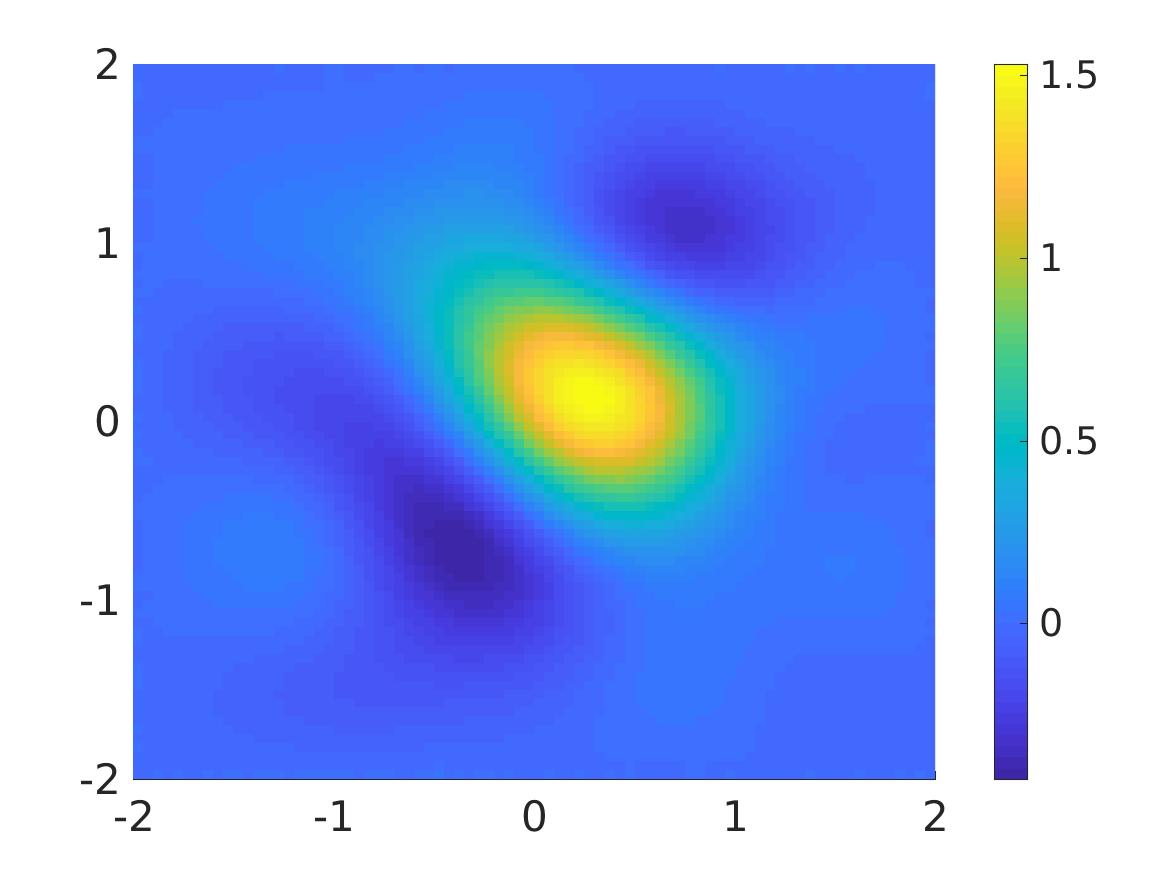} 
		}
	\caption{\label{test 1} Test 1. The true and computed source functions.  Our method still  works well when $\delta = 100\%.$ It is shown in (e) that the reconstructed value of $f_{\rm comp}$ with $\delta = 75\%$ is quite accurate, even better than in (d), but in contrast, the reconstructed shape starts to break down.}
	\end{center}
	\end{figure}
	
\begin{table}[h!]
{\small
\caption{\label{table 1} Test 1. Correct and computed maximal values of source functions. ${\rm error}_{\rm rel}$ denotes the relative error of the reconstructed maximal value. ${\rm pos}_{\rm true}$ is the true position of the inclusion; i.e., the maximizer of $f_{\rm true}$. ${\rm pos}_{\rm comp}$ is the computed position of the inclusion. ${\rm dis}_{\rm err}$ is the absolute error of the reconstructed positions.}
\begin{center}
\begin{tabular}{|c|c|c|c|c|c|c|}
\hline
 noise level &  $\max f_{\rm true}$ & $\max f_{\rm comp}$ &${\rm error}_{\rm rel}$&${\rm pos}_{\rm true}$&${\rm pos}_{\rm comp}$&${\rm dis_{\rm err}}$\\ 
\hline
0\%& 1 &0.99 &1.0\%&$(0,0)$ &$(0,0)$&0\\
\hline
25\%& 1 &0.96 &4.0\%&$(0,0)$& $(-0.05, 0)$ &0.05 \\ 
\hline
50\% &1&1.21 &21.0\%&$(0,0)$& $(-0.05, 0.1)$&0.11
\\
\hline
75\% &1&1.01 &1.0\%&$(0,0)$&$( 0.2, -0.1)$&0.22
\\
\hline
100\% &1&1.53 &53.0\%&$(0,0)$&$( 0.25, 0.1)$&0.27
\\
\hline
\end{tabular}%
\end{center}
}
\end{table}
 It is evident that our method is robust for Test 1 in the sense that the reconstructed maximal value of the function $f$ and the reconstructed shape and position of the inclusion are quite accurate.
 \item {\it Test 2. The case of two inclusions.}   The function $f_{\rm true}$ is a smooth function supported in two disks with radius $r = 0.8$ centered at $\xx_1 = (-1, 0)$ and $\xx_2 = (1, 0)$ respectively. The function $f_{\rm true}$ is given by the formula 
\[
	f(\xx) = \left\{
		\begin{array}{ll}
			\ds e^{-\frac{r^2}{r^2 - |\xx - \xx_1|^2} + 1} & \mbox{if } |\xx - \xx_1| < r,\\ 
			\ds e^{-\frac{r^2}{r^2 - |\xx - \xx_2|^2} + 1} & \mbox{if } |\xx - \xx_2| < r,\\
			0&\mbox{otherwise}.
		\end{array}
	\right.
\]
	Figure \ref{test 2} displays the functions $f_{\rm true}$ and $f_{\rm comp}$. 
	Table \ref{table 2} show the reconstructed value of the function $f_{\rm comp}$ and the relative error.
	The noise levels are $\delta = 0\%$, $25\%$, $50\%$, $75\%$ and $100\%.$
	\begin{figure}[h!]
	\begin{center}
		\subfloat[The function $f_{\rm true}$]{
			\includegraphics[width = .3\textwidth]{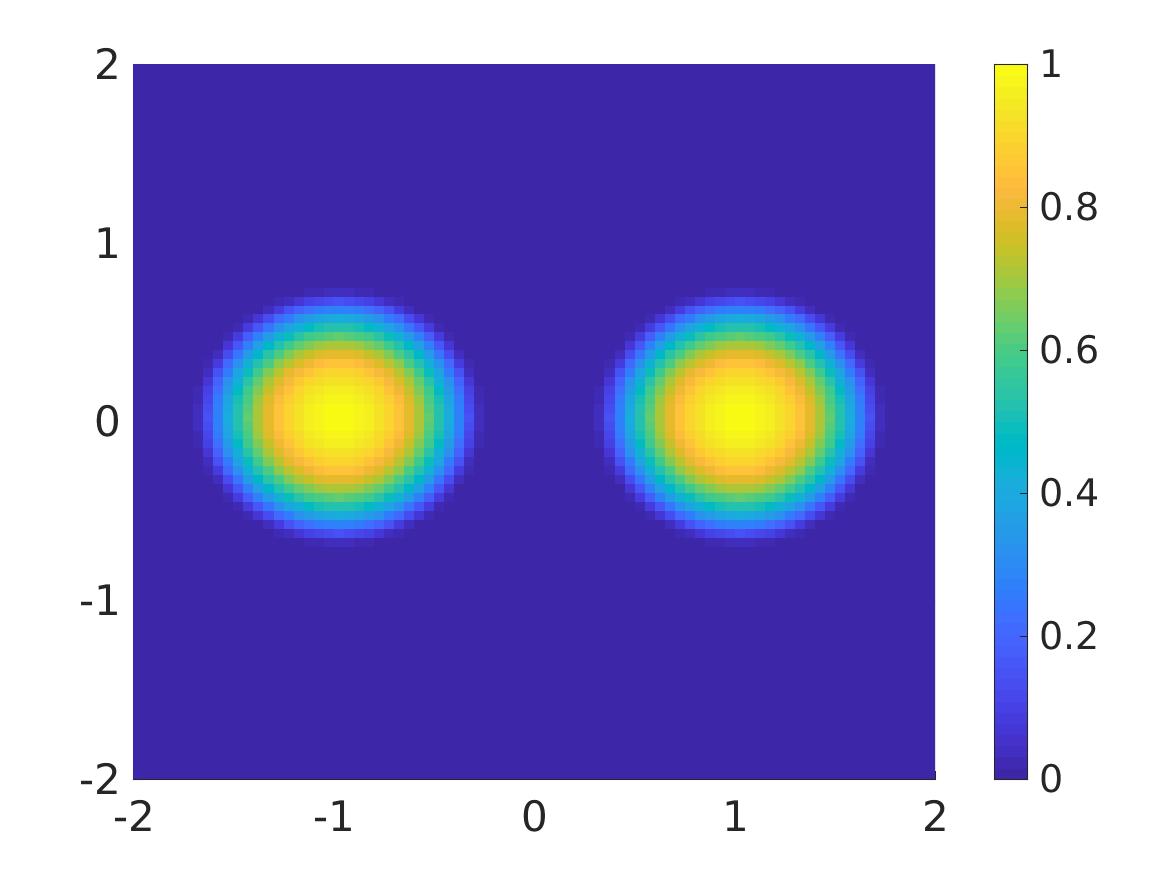} 
		} \hfill
		\subfloat[$f_{\rm comp}$, $\delta = 0\%$]{
			\includegraphics[width = .3\textwidth]{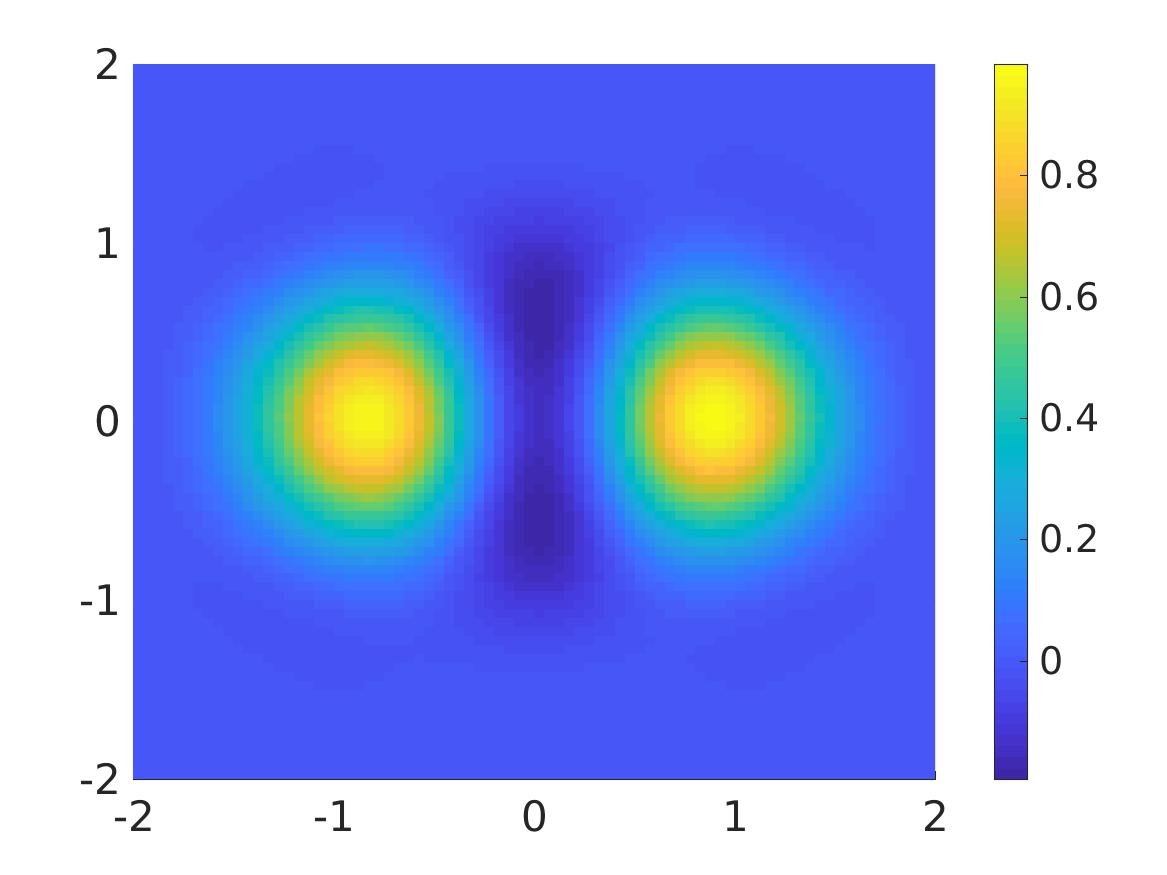} 
		}\hfill
		\subfloat[$f_{\rm comp}$, $\delta = 25\%$]{
			\includegraphics[width = .3\textwidth]{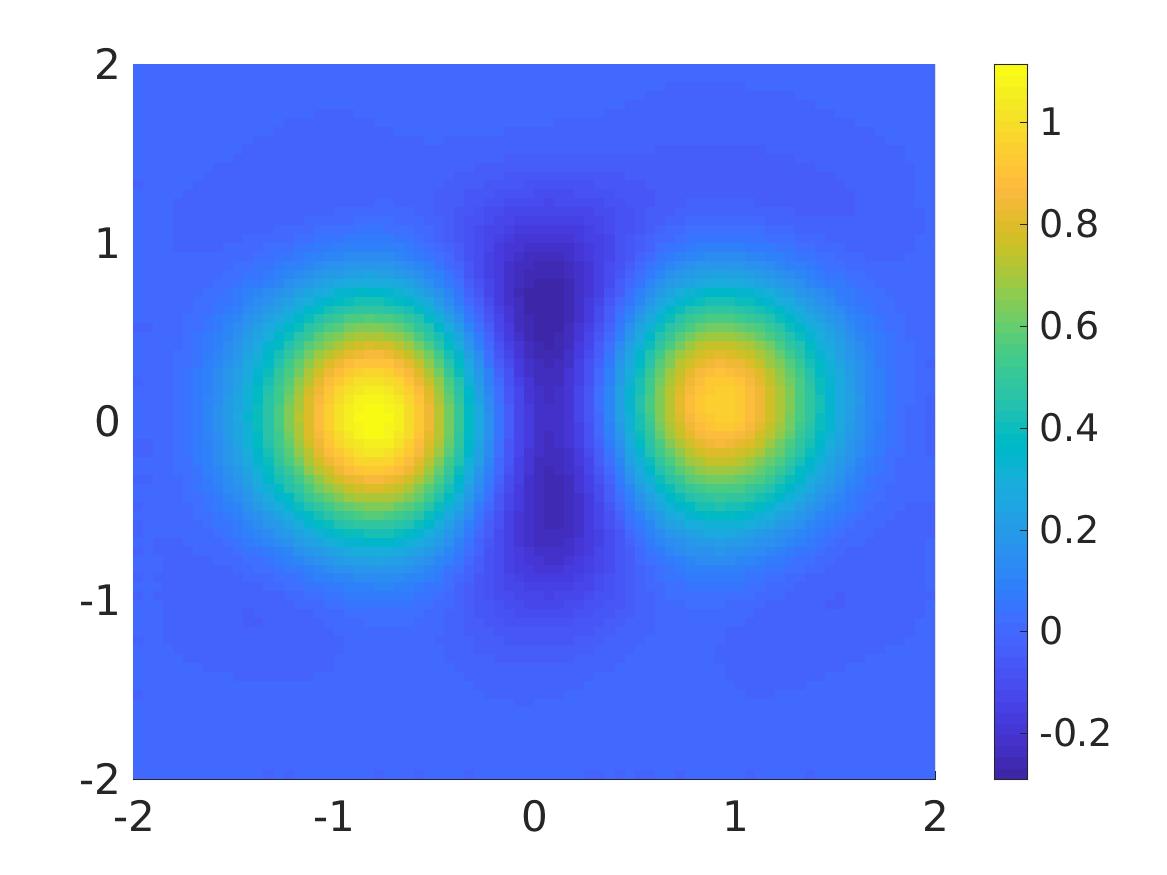} 
		}
		
		\subfloat[$f_{\rm comp}$, $\delta = 50\%$]{
			\includegraphics[width = .3\textwidth]{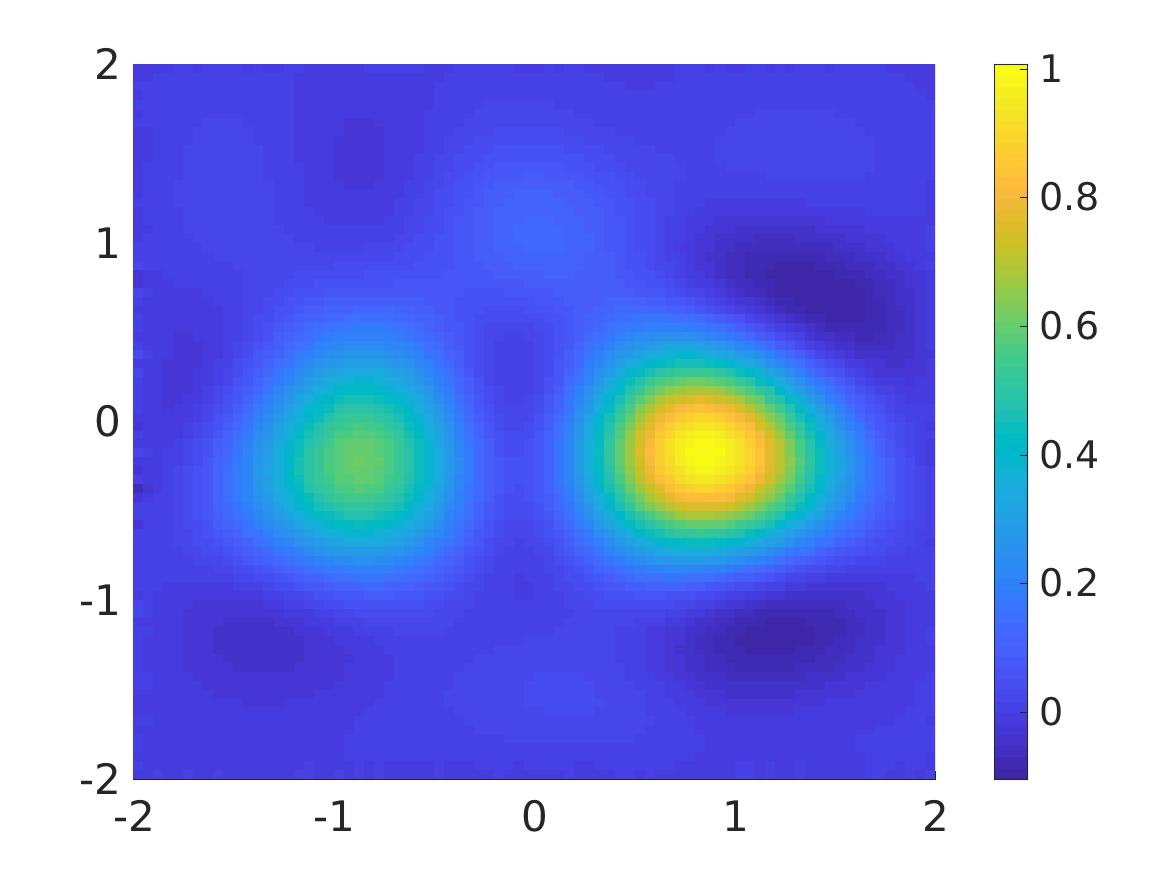} 
		}\hfill
		\subfloat[$f_{\rm comp}$, $\delta = 75\%$]{
			\includegraphics[width = .3\textwidth]{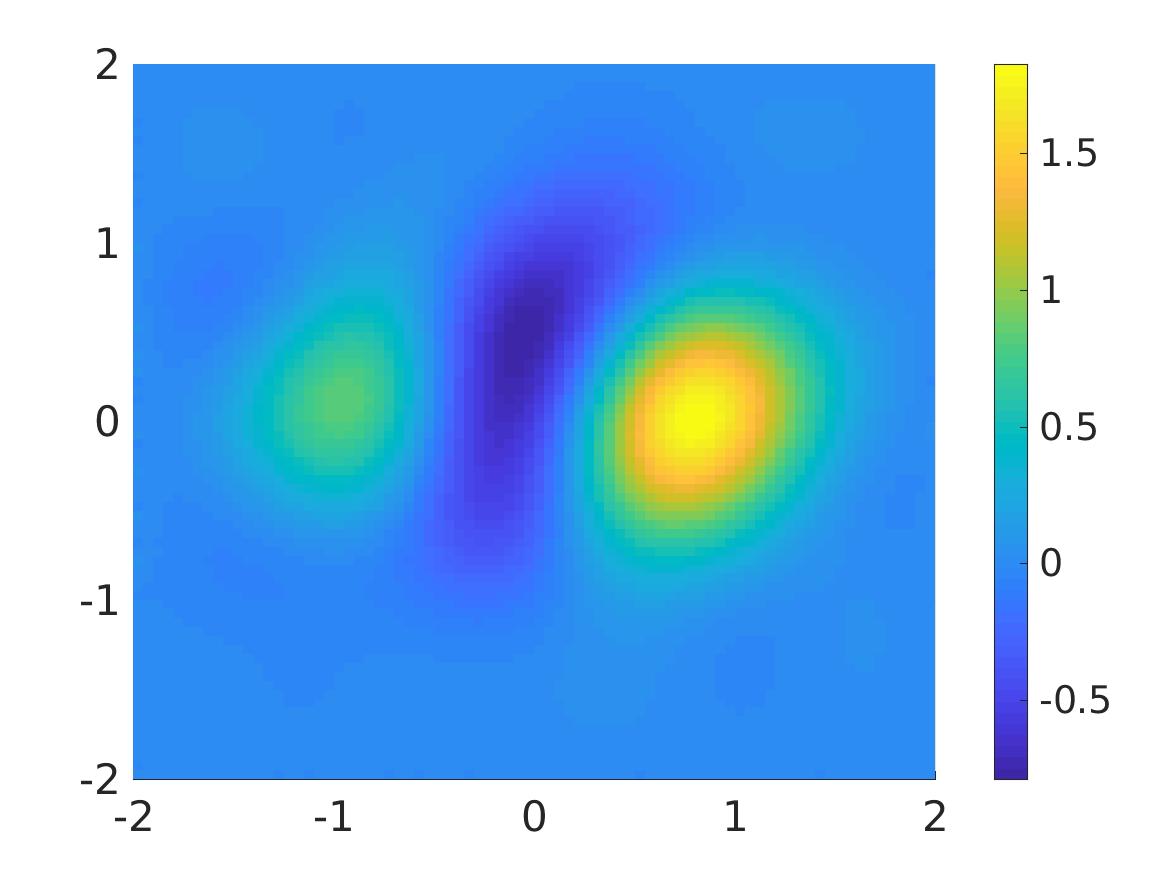} 
		}\hfill
		\subfloat[$f_{\rm comp}$, $\delta = 100\%$]{
			\includegraphics[width = .3\textwidth]{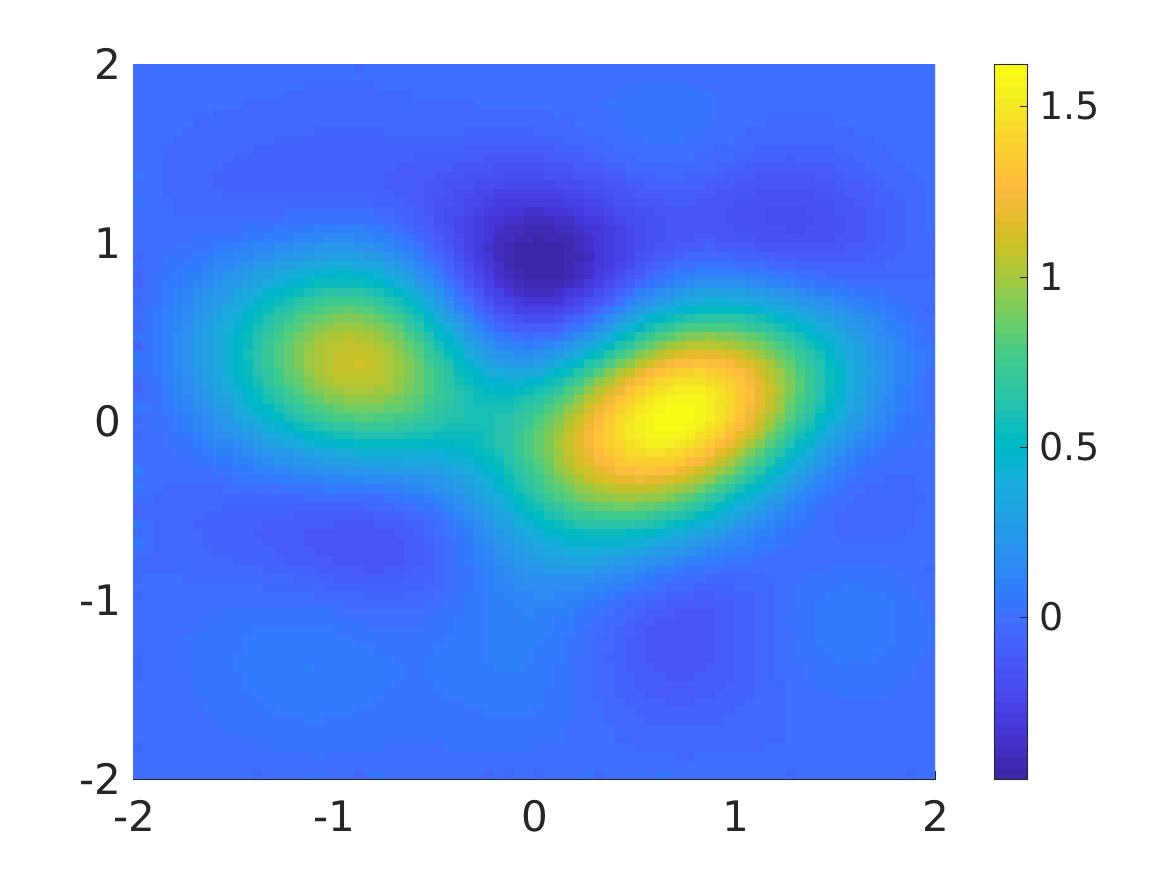} 
		}
	\caption{\label{test 2} Test 2. The true and computed source functions. The reconstruction of the two inclusions are not symmetric probably because the true function $c$, see Figure \ref{fig c} for its graph, is negative on the left and positive on the right. However, both inclusions can be seen when the noise level goes up to $100\%.$}
	\end{center}
	\end{figure}

\begin{table}[h!]
{\small
\caption{\label{table 2} Test 2. Correct and computed maximal values of the inclusions. 
$\max_{\rm inc, true}$ and $\max_{\rm inc, comp}$ are the true and computed, respectively, maximal values of the source in an inclusion.
${\rm error}_{\rm rel}$ denotes the relative error of the reconstructed maximal value. ${\rm pos}_{\rm true}$ is the true position of the inclusion; i.e., the maximizer of $f_{\rm true}$. ${\rm pos}_{\rm comp}$ is the computed position of the inclusion. ${\rm dis}_{\rm err}$ is the absolute error of the reconstructed positions.}
\begin{center}
\begin{tabular}{|c|c|c|c|c|c|c|c|}
\hline
 noise level &inclusion&  $\max_{\rm inc, true}$ &  $\max_{\rm inc, comp}$ & ${\rm error}_{\rm rel}$&${\rm pos}_{\rm true}$&${\rm pos}_{\rm comp}$&${\rm dis_{\rm err}}$
 \\ 
 \hline
 0\%&left&1&0.96&4.0\%&$(-1,0)$&$(-0.85,0)$&0.15\\
 \hline
 0\%&right&1&0.98&2.0\%&$(1,0)$&$(0.85,0)$&0.15\\
 \hline
 25\%&left&1&1.11&11\%&$(-1, 0)$&$(-0.85, 0)$&0.15\\
 \hline
  25\%&right&1&0.96&4\%&$(1, 0)$&$(0.9, 0.1)$&0.14\\
 \hline
 50\%&left&1&0.61&39\%&$(-1, 0)$&$(-0.9, 0.25)$&0.27\\
 \hline
 50\%&right&1&1.01&1\%&$(1, 0)$&$(0.85, -0.2)$&0.25\\
 \hline
 75\%&left&1&0.84&26\%&$(-1, 0)$&$(-1, 0.1)$&0.1\\
 \hline
  75\%&right&1&1.82&82\%&$(1, 0)$&$(0.8, 0)$&0.2\\
  \hline
   100\%&left&1&1.1&10\%&$(-1, 0)$&$(-0.9, 0.3)$&0.32\\
 \hline
 100\%&right&1&1.58&58\%&$(1, 0)$&$(0.05, 0.8)$&0.21\\
 \hline
\end{tabular}
\end{center}
}
\end{table}

The reconstruction in Test 2 is good. In this test, the reconstruct breaks down when the noise level is $75\%$ although we are able to detect the inclusions with higher noise levels.
\item {\it Test 3. The case of non-inclusion and nonsmooth function.}   The function $f_{\rm true}$ is the characteristic function of the letter $Y$.
	Figure \ref{test 3} displays the functions $f_{\rm true}$ and $f_{\rm comp}$. 
	The noise levels are $\delta = 10\%$ and $15\%.$
	\begin{figure}[h!]
	\begin{center}
		\subfloat[The function $f_{\rm true}$]{			\includegraphics[width = .3\textwidth]{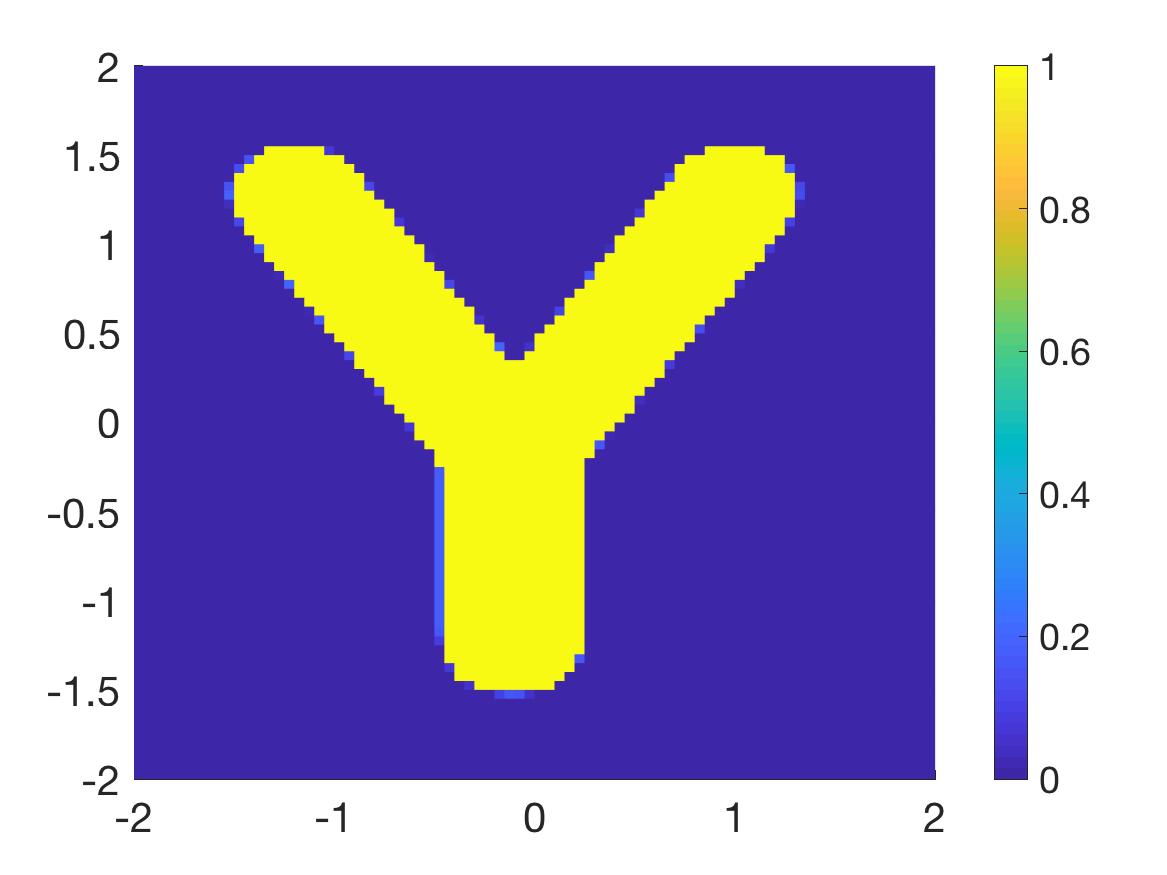} 
		} \hfill
		\subfloat[$f_{\rm comp}$, $\delta = 10\%$]{			\includegraphics[width = .3\textwidth]{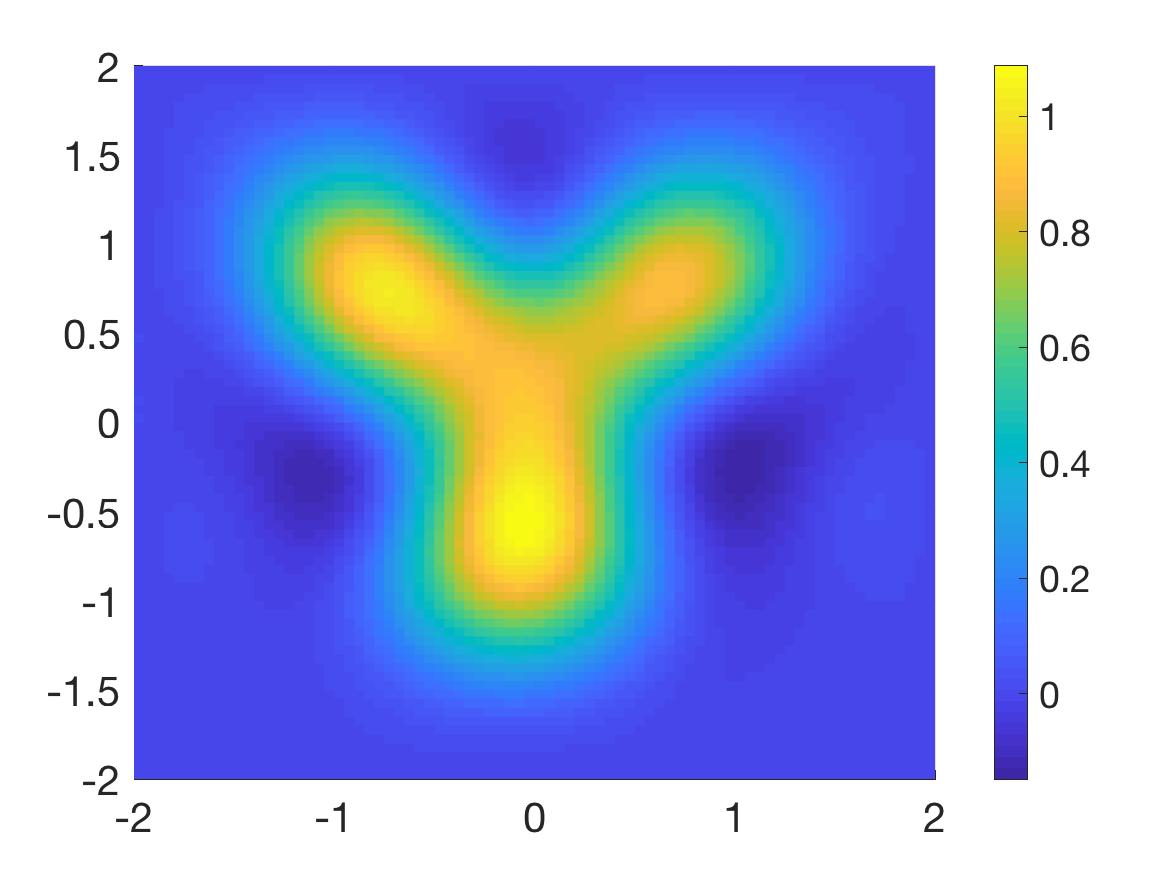} 
		}\hfill
		\subfloat[$f_{\rm comp}$, $\delta = 15\%$]{			\includegraphics[width = .3\textwidth]{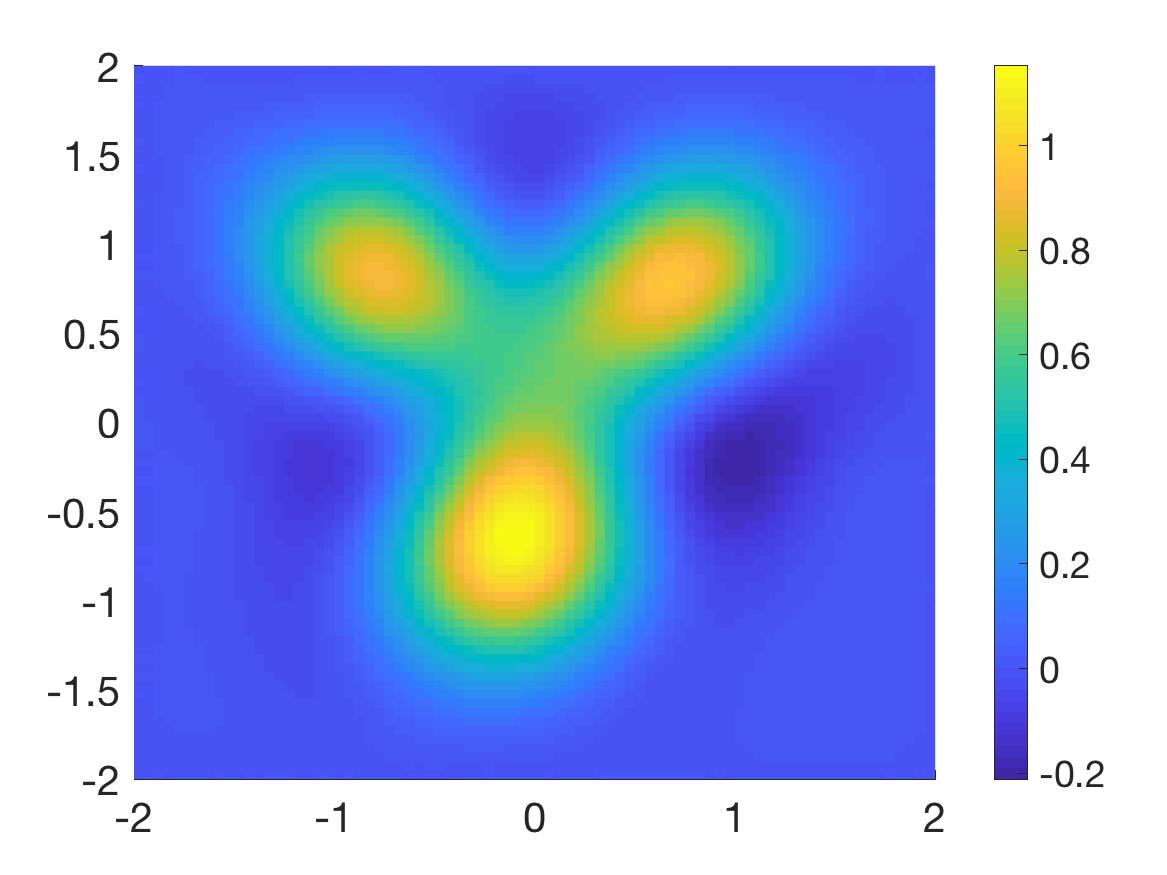} 
		}
	\caption{\label{test 3} Test 3. The true and computed source functions. The letter $Y$ can be detected well in this case. The true maximal value of $f_{\rm true}$ is 1. The computed maximal value of $f_{\rm comp}$ when $\delta = 10\%$ is 1.09 (relative error 9\%). The computed maximal value of $f_{\rm comp}$ when $\delta = 15\%$ is 1.15 (relative error 15\%).}
	\end{center}
	\end{figure}

 We can reconstruct the letter $Y$ and the reconstructed maximal of $f_{\rm comp}$ is good when $\delta = 10\%$ but the error is large when the noise level reaches $15\%.$

\item {\it Test 4. The case of non-inclusion and nonsmooth function.}   The function $f_{\rm true}$ is the characteristic function of the letter $\lambda$.
	Figure \ref{test 4} displays the functions $f_{\rm true}$ and $f_{\rm comp}$. 
	The noise levels are $\delta = 10\%$ and $15\%.$
	\begin{figure}[h!]
	\begin{center}
		\subfloat[The function $f_{\rm true}$]{			\includegraphics[width = .3\textwidth]{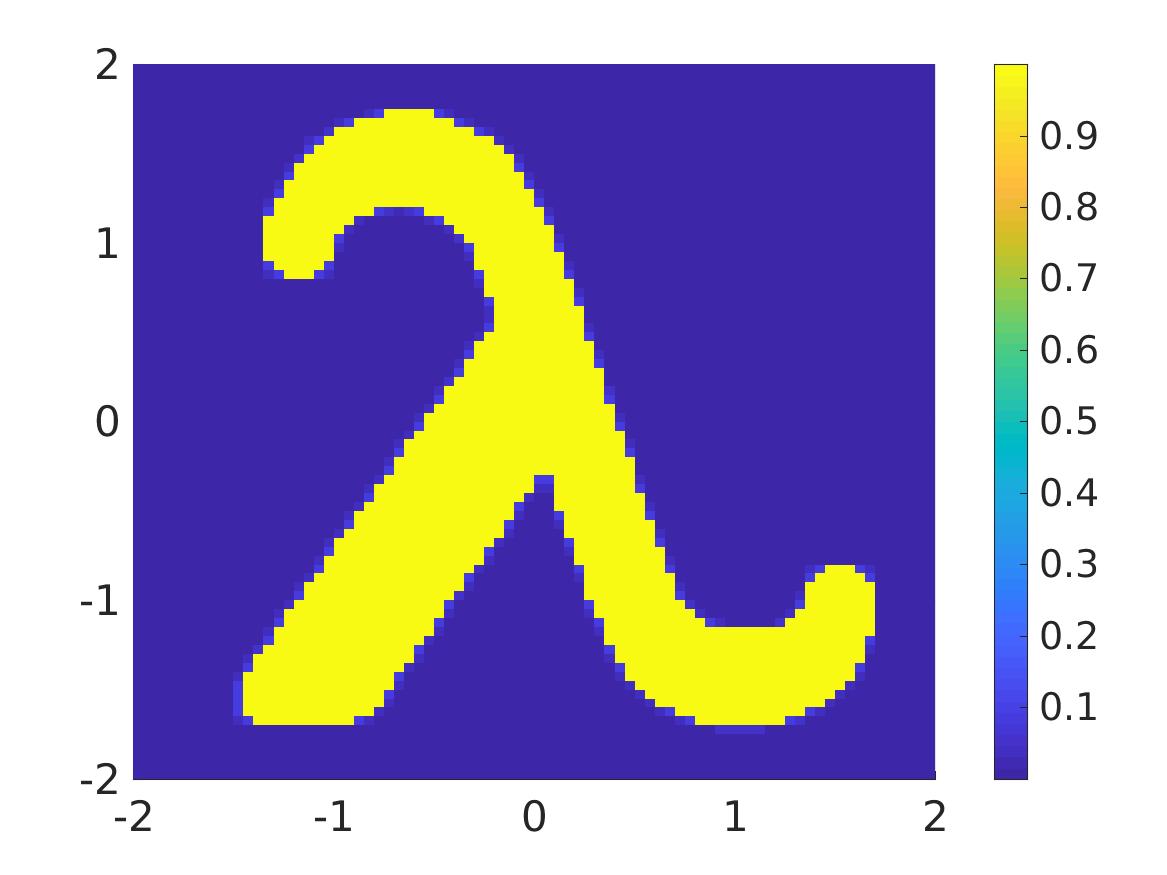} 
		} \hfill
		\subfloat[\label{6b}$f_{\rm comp}$, $\delta = 10\%$]{			\includegraphics[width = .3\textwidth]{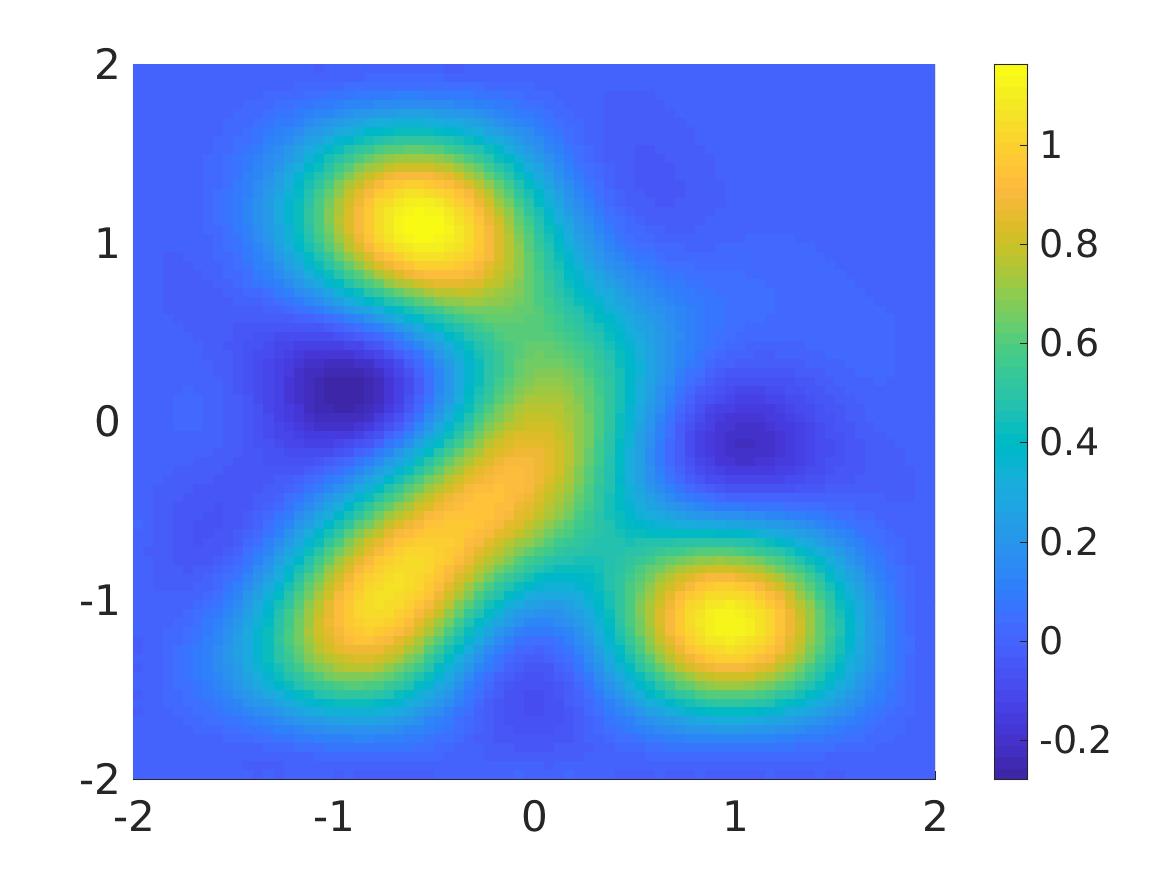} 
		}\hfill
		\subfloat[\label{6c}$f_{\rm comp}$, $\delta = 15\%$]{			\includegraphics[width = .3\textwidth]{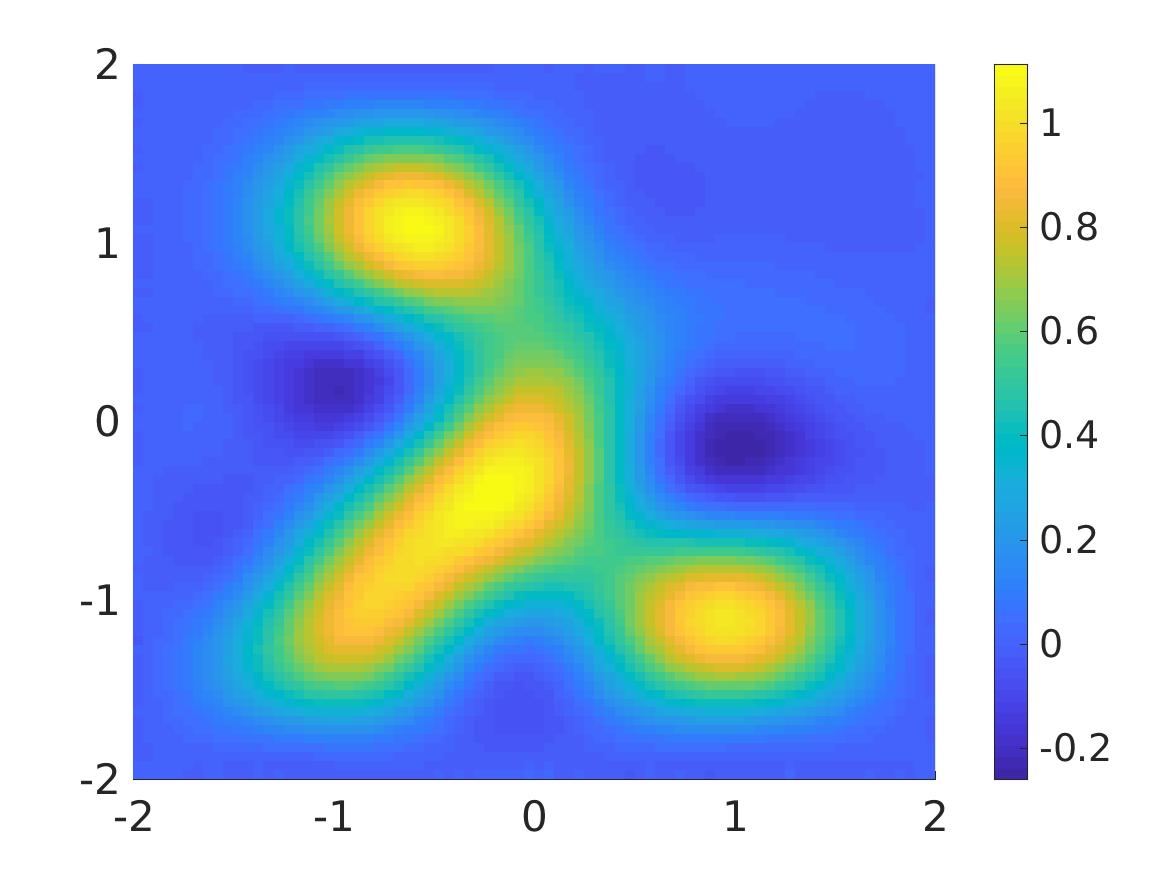} 
		}
	\caption{\label{test 4} Test 4. The true and computed source functions. The reconstruction of $\lambda$ is acceptable.
	The true maximal value of $f_{\rm true}$ is 1. The computed maximal value of $f_{\rm comp}$ when $\delta = 10\%$ is 1.16 (relative error 16\%). The computed maximal value of $f_{\rm comp}$ when $\delta = 15\%$ is 1.11 (relative error 11\%).}
	\end{center}
	\end{figure}

The image of $\lambda$ in Test 4 is acceptable. The reconstructed maximal value in Figure \ref{6c} is better than that in Figure \ref{6b} but the reconstruction of $\lambda$ in Figure \ref{6c} is not as good as that in Figure \ref{6b}.
\end{enumerate}

\section{Concluding remarks}

In this paper, we have solved the problem of reconstructing the initial condition of solution to a general class of parabolic equation from the measurement of lateral Cauchy data.
The main points of the method is derive an approximate model by a truncation of the Fourier series with respect to a special basis.
We solved the approximation model by the quasi-reversibility method.
The convergence of this method when the noise tends to $0$ was proved.
More importantly, numerical examples show that our method is robust when proving accurate reconstructions of the unknown source function from highly noisy data.

Although our method leads to good numerical results, it has a drawback. 
The proof of the ``convergence" of the system \eqref{2.10} as $N \to \infty$ is challenging and is omitted in this paper.
We refer the reader to \cite[Section 4]{Klibanov:anm2015} for an alternative approach to solve Problem \ref{ISP} by which we can avoid this non-rigorousness. 
This method is based on the Carleman estimate for parabolic operators.
However, in this case we can determine a ``near" initial condition for the function $u(\xx, t)$. That means, we can recover the function $u(\xx, \epsilon)$  where $\epsilon$ is any small number. 
Implementation for the method in \cite[Section 4]{Klibanov:anm2015} is valuable. 
We reserve it for a future reseach.

\begin{center}
\textbf{Acknowledgments}
\end{center}
The authors are grateful to Michael V. Klibanov for many fruitful discussions.
The work of Nguyen was supported by US Army Research Laboratory
and US Army Research Office grant W911NF-19-1-0044. In addition, the effort
of Nguyen and Li was supported by research funds FRG 111172 provided by The
University of North Carolina at Charlotte.


\end{document}